\documentclass[12pt]{amsart}
\usepackage{ amsmath, amsthm, amsfonts, amssymb, color}
 \usepackage{mathrsfs}
\usepackage{amsfonts, amsmath}
 \usepackage{amsmath,amstext,amsthm,amssymb,amsxtra}
 \usepackage{txfonts} 
 \usepackage[colorlinks, citecolor=blue,pagebackref,hypertexnames=false]{hyperref}
 \allowdisplaybreaks
 \usepackage{pgf,tikz}

\begin{document}

 \baselineskip 16.6pt
\hfuzz=6pt

\widowpenalty=10000

\newtheorem{cl}{Claim}
\newtheorem{theorem}{Theorem}[section]
\newtheorem{proposition}[theorem]{Proposition}
\newtheorem{coro}[theorem]{Corollary}
\newtheorem{lemma}[theorem]{Lemma}
\newtheorem{definition}[theorem]{Definition}
\newtheorem{assum}{Assumption}[section]
\newtheorem{example}[theorem]{Example}
\newtheorem{remark}[theorem]{Remark}
\renewcommand{\theequation}
{\thesection.\arabic{equation}}

\def\SL{\sqrt H}

\newcommand{\mar}[1]{{\marginpar{\sffamily{\scriptsize
        #1}}}}

\newcommand{\as}[1]{{\mar{AS:#1}}}

\newcommand\R{\mathbb{R}}
\newcommand\RR{\mathbb{R}}
\newcommand\CC{\mathbb{C}}
\newcommand\NN{\mathbb{N}}
\newcommand\ZZ{\mathbb{Z}}
\def\RN {\mathbb{R}^n}
\renewcommand\Re{\operatorname{Re}}
\renewcommand\Im{\operatorname{Im}}

\newcommand{\mc}{\mathcal}
\newcommand\D{\mathcal{D}}
\def\hs{\hspace{0.33cm}}
\newcommand{\la}{\alpha}
\def \l {\alpha}
\newcommand{\eps}{\varepsilon}
\newcommand{\pl}{\partial}
\newcommand{\supp}{{\rm supp}{\hspace{.05cm}}}
\newcommand{\x}{\times}
\newcommand{\lag}{\langle}
\newcommand{\rag}{\rangle}

\newcommand\wrt{\,{\rm d}}
\newcommand{\botimes}{\bar{\otimes}}

\title[]{Campanato spaces via quantum semigroups}

\author{Guixiang Hong, Yuanyuan Jing and Ping Li}

 \address{Guixiang Hong, Institute for Advanced Study in Mathematics, Harbin Institute of Technology, 150001, Harbin, P.R. China}
\email{gxhong@hit.edu.cn}
 \address{Yuanyuan Jing, Institute for Advanced Study in Mathematics, Harbin Institute of Technology, 150001, Harbin, P.R. China}
\email{yuanyuanjing@whu.edu.cn}
\address{Ping Li, School of Mathematics and Systems Science, and Center for Mathematics Sciences, Wuhan University of Science and Technology, 430065, Wuhan, P.R. China}
\email{liping@whu.edu.cn}

\date{\today}

\begin{abstract}
In this paper, we continue to investigate Campanato spaces via semigroups on von Neumann algebras. One of the main results is their coincidence with Lipschitz spaces for {\it all} regularity indices $\alpha > 0$ and for all {\it analytic} semigroups on von Neumann algebras {\it not necessarily being finite}; the desired self-improving property of Campanato spaces has also been verified. As a consequence, we demonstrate that the column Campanato spaces are isomorphic to the row ones for {\it all} $\alpha > 0$. This not only removes several restrictions in the previous work \cite{HJ24} by the first two authors, but also extends the previous results to any analytic semigroup, and thus resolves several problems left open in \cite{HJ24}. Both the results and the proof are new even in the commutative setting.
\end{abstract}

\maketitle
\section{Introduction}
\hskip\parindent
The inhomogeneous Campanato space with regularity index $\alpha\geq0$ on $\mathbb{R}^{n}$ (see e.g. \cite[Page 83]{s79} and \cite[Page 105]{JTW83}) is defined as a subset of $L^{\infty}(\mathbb{R}^{n})$ with the norm given by
\begin{equation}\label{c1}
 \|f\|_{\mathcal{L}_{\alpha}(\mathbb{R}^{n})}=\|f\|_{\infty}+\sup_{B}\frac{1}{|B|^{\frac{\alpha}{n}}}\left(\frac{1}{|B|}\int_{B}|f(x)-P_{B}f|^{2}dx\right)^{\frac{1}{2}},
\end{equation}
where the supremum is taken over all balls $B$ in $\mathbb{R}^{n}$, and $P_{B}f$ is the unique polynomial of degree at most $[\alpha]$---the integer part of $\alpha$---such that
$$\int_{B}[f(x)-P_{B}f(x)]x^{\theta}dx=0$$
for all multi-indices $\theta$ with $0\leq|\theta|\leq[\alpha]+1.$ It is well-known that the homogeneous part in \eqref{c1} gives the classical BMO norm when $\alpha=0$ on one hand, and it is equivalent to the homogeneous Lipschitz norm via differences when $\alpha>0$ (see e.g. \cite[Theorem 5.30]{g85} and \cite{TW80}) on the other hand. This in turn yields the characterization of inhomogeneous Campanato spaces in terms of Stein's Lipschitz spaces defined by the classical Poisson semigroup (see e.g. \cite[Page 142]{s70} or \eqref{l1} below).

In order to investigate the Campanato spaces on homogeneous spaces where the notion `polynomials' is not available, Doung and Yan \cite{duy05,duya05,dy09} introduced the homogeneous Campanato spaces associated to approximation identity $(A_t)_{t>0}$ satisfying  $m$-order Poisson upper bound by setting
\begin{equation*}
 P_{B}f(x)=A_{m,t_{B}}f(x)=[I-(I-A_{t_{B}})^{[\frac{\alpha}{m}]+1}]f(x)
\end{equation*}
 for a certain range of $\alpha$ depending on $m$, where $t_B$ is the radius of ball $B$. However, when $(A_{t})_{t>0}$ is the classical Poisson semigroup, it seems that they did not mention its equivalence with the homogeneous part of \eqref{c1}. While one of our main results---Theorem \ref{e25}---will deduce this equivalence, we also verify the equivalence for $0<\alpha<1/2$  in the appendix by exploiting basic properties of Poisson semigroup such as kernel estimates and Markovian property etc.. See Theorem \ref{c1c} and the subsequent remark.

Very recently, in the noncommutative framework where the geometric objects such as balls or cubes are absent, the first two authors introduced Campanato spaces defined purely by semigroups in  \cite{HJ24} with partial motivation from Mei's semigroup BMO theory  \cite{m08, mj12, tms19}. One of the main results of that paper is the isomorphism between the newly defined semigroup Campanato spaces \eqref{p1} and Stein's semigroup Lipschitz spaces \eqref{l1}, which are associated with those semigroups subordinated to quantum Markov semigroups. This in turn yields a striking new phenomenon which is usually not expected in noncommutative analysis, that is, the column Campanato spaces are isomorphic to the row ones for small $\alpha\neq0$. The self-improving property of Campanato spaces has also been obtained. However, those results have several drawbacks (see \cite[Remark 1.5]{HJ24}): the isomorphisms and self-improving property held only for Campanato spaces with regularity index $0<\alpha<2$; the underlying von Neumann algebra had to be finite; except its subordinated semigroups, it had been unclear whether the results held true or not for the quantum Markov semigroup itself.

In the present paper, we are able to remedy all the above-mentioned drawbacks and thus generalize the previous work \cite{HJ24} significantly. That is, we are going to establish the isomorphisms and self-improving property for all $\alpha>0$ and for any analytic semigroup (including the subordinated semigroup ones) on von Neumann algebras not necessarily being finite. Moreover, when going back to the classical setting, these results provide new characterizations or properties of Campanato spaces on $\mathbb R^n$, see Theorem \ref{c1c} and Proposition \ref{c2c} in the appendix for more details.

\smallskip

Let us state the main results precisely with the relavant notions given in the preliminary section. Let $\mathcal T=(T_{t})_{t>0}$ be a positive semigroup acting on a von Neumann algebra $\mathcal M$. Let $\alpha>0$. The column Campanato space $\mathcal{L}^{c}_{\alpha}(\mathcal{T})$ is defined as a subspace of $\mathcal M$ with a finite norm given by
 \begin{equation}\label{p1}
   \|f\|_{\mathcal{L}^{c}_{\alpha}(\mathcal{T})}= \|f\|_{\infty}+\sup_{t>0}\frac{1}{t^{\alpha}}\left\|T_{t}|(I-T_{t})^{[\alpha]+1}f|^{2}\right\|^{\frac{1}{2}}_{\infty}.
 \end{equation}
The row space $\mathcal{L}_{\alpha}^{r}(\mathcal T)$ is the space of all $f\in\mathcal M$ such that $f^{\ast}\in\mathcal{L}_{\alpha}^{c}(\mathcal T)$, equipped with the norm
 $\left\|f\right\|_{\mathcal L_{\alpha}^{r}(\mathcal T)}=\left\|f^{*}\right\|_{\mathcal L_{\alpha}^{c}(\mathcal T)}$. The mixture space $\mathcal{L}_{\alpha}^{cr}(\mathcal T)$ is defined as $\mathcal{L}_{\alpha}^{c}(\mathcal T)\cap \mathcal{L}_{\alpha}^{r}(\mathcal T)$ equipped with the intersection norm.
On the other hand, drawing inspiration from Stein \cite[Section 4.2]{s70}, 
the semigroup Lipschitz space ${\Lambda_{\alpha}(\mathcal{T})}$ is defined as the subspace of $\mathcal M$ with norm given by
\begin{equation}\label{l1}
 \|f\|_{\Lambda_{\alpha}(\mathcal{T})}=\|f\|_{\infty}+\sup_{t>0}\frac{1}{t^{\alpha-([\alpha]+1)}}\left\|\frac{\partial^{[\alpha]+1}T_{t}f}{\partial t^{[\alpha]+1}}\right\|_{\infty}.
\end{equation}
It is easy to derive $\left\|f^{*}\right\|_{\Lambda_{\alpha}({\mathcal T)}}=\left\|f\right\|_{\Lambda_{\alpha}({\mathcal T})}$ from the positivity of the semigroup.


The following is the first main result.

\begin{theorem}\label{e25}
Let $\mathcal T=(T_{t})_{t>0}$ be a 2-positive and contractive analytic semigroup acting on a von Neumann algebra $\mathcal M$ and $\alpha>0$. Then for $f\in \mathcal{M}$, we have
\begin{equation}\label{b2}
 \left\|f\right\|_{\mathcal L_{\alpha}^{c}(\mathcal T)}\simeq_{\alpha}\left\|f\right\|_{\Lambda_{\alpha}(\mathcal T)}.
\end{equation}
A similar assertion is true for the row Campanato spaces.
\end{theorem}

A consequence is the following

\begin{coro}\label{jy}
Let $\mathcal T=(T_{t})_{t>0}$ be as above and $\alpha>0$. Then we have the following isomorphisms with equivalent norms,
\begin{equation}\label{jy0}
  \mathcal{L}_{\alpha}^{c}(\mathcal T)=\mathcal{L}_{\alpha}^{r}(\mathcal T)=\mathcal L_{\alpha}^{cr}(\mathcal T).
\end{equation}
\end{coro}

As mentioned above, these results were known previously only when $0<\alpha<2$ and the semigroup $\mathcal T$ is subordinated to some Markov semigroup on finite von Neumann algebra \cite{HJ24}. Analytic semigroups form a very large class of semigroups besides the subordinated ones (see e.g. \cite[Appendix]{xu24}) which might admit fertile applications; a similar phenomenon has been observed in Xu's work \cite{xu24} on vector-valued Littlewood-Paley-Stein theory.  

In \cite[Open Problem~(P.5)]{g19}, Gonz\'alez-P\'erez asked whether the homogeneous semigroup H\"older spaces admit an equivalent mixed row--column Campanato characterization. The spaces there are defined using Poisson semigroups subordinated
to Markov semigroups. Such Poisson semigroups are unital and completely positive, hence $2$-positive and contractive on $\mathcal{M}$; they are also analytic on $\mathcal{M}$ by \cite[Remark~3.2]{g19}. Thus, they satisfy the hypotheses of Theorem~\ref{e25}. For $0<\alpha<1$, when $\mathcal{T}$ is such a Poisson semigroup, the seminorm parts of \eqref{p1} and \eqref{l1} coincide with the corresponding Campanato and H\"older seminorms in \cite{g19}. The desired equivalence of homogeneous seminorms follows from Theorem~\ref{e25}. This answers the norm-equivalence question in (P.5). Moreover, Theorem~\ref{e25} extends this characterization to all $\alpha>0$ for general $2$-positive contractive analytic semigroups.

It is \eqref{b2} that provides a new characterization of classical Campanato or Lipschitz spaces by letting $(T_{t})_{t>0}$ be the Poisson semigroup on $\mathbb R^n$. Indeed, as recalled at the beginning of the introduction, Stein's semigroup Lipschitz norm is equivalent to the classical inhomogeneous Campanato norm, which by \eqref{b2}, is equivalent to the one \eqref{p1} defined purely via Poisson semigroup for all $\alpha>0$. This seems to have been unknown in the literature before, while the cases $0<\alpha<1/2$ could be concluded by the classical approach (see the Remark after Theorem \ref{c1c}). It should be pointed out that when $(T_{t})_{t>0}$ is the heat semigroup on $\mathbb R^n$, the equivalence \eqref{b2} is well-known, as can be seen from the results in \cite{ddy05,dy09}, \cite[Theorem 6.3.7]{gl09} and \cite[Corollary 3 (i)]{t82}. One can find more characterizations on classical Campanato or Lipschitz spaces in Appendix A. Moreover, our Theorem \ref{e25} also applies to the Poisson semigroup subordinated to the Ornstein-Uhlenbeck semigroup on $\mathbb R^n$ equipped with the Gaussian measure, and more information on other related spaces and characterizations can be found in Appendix B.

\begin{remark}
{\rm  (i) Let us mention again that Corollary \ref{jy} presents a surprising phenomenon in noncommutative analysis since it is well-known that the isomorphisms \eqref{jy0} may not hold for $\alpha=0$ (see, for instance, \cite[Remark 2.11]{chlm}, \cite[Appendix]{px97}, \cite[Appendix B]{gjp21}, \cite{Mei07} or \cite[Page 6]{jaje23}).

(ii) If additionally $(T_{t})_{t>0}$ is assumed to be completely positive, then one can deduce that the isomorphisms \eqref{jy0} can be automatically improved to complete isomorphisms,
 \begin{align}\label{js}
 \|g\|_{M_{n}( \mathcal L_{\alpha}^{r}(\mathcal T))}=\|g^{*}\|_{M_{n}( \mathcal L_{\alpha}^{c}(\mathcal T))}\simeq\|g\|_{M_{n}( \mathcal L_{\alpha}^{c}(\mathcal T))},
\end{align}
where $\mathcal L_{\alpha}^{\dag}(\mathcal T)$ is equipped with the operator space structure by replacing respectively $\mathcal M$ and $T_{t}$ with $M_{n}\otimes\mathcal M$ and $I_n\otimes T_t$ such that $\|g\|_{M_{n}( \mathcal L_{\alpha}^{\dag}(\mathcal T))}=\|g\|_{\mathcal L_{\alpha}^{\dag}((I_n\otimes T_t )_{t>0})}$.
Here $\dag=\{c,r,cr\}$ and $M_n$ is the algebra of $n\times n$ complex matrices and $I_n$ is the identity map on $M_n$.}
\end{remark}

The new ingredients that allow us to conclude Theorem \ref{e25} beyond the restrictions $ 0 < \alpha <2$ and finite von Neumann algebras in \cite{HJ24} are the observations appearing in Lemmas \ref{190} and \ref{531}. These two lemmas enable us to bypass the use of the regeneration formulas (4.12) and (4.15) in \cite{HJ24}, which not only depend on the finiteness of the von Neumann algebra but also induce the restriction $0 < \alpha <2$. On the other hand, the main feature of analytic semigroup utilized in the present paper is Proposition \ref{wss}, which derives from holomorphic functional calculus.

\smallskip

The above new ingredients also allow us to obtain the self-improving property of the semigroup Campanato spaces for all $\alpha>0$, which might admit further applications to spectral multipliers and Schr\"odinger groups as in the case of BMO spaces (cf. \cite{sgd11,hm09, DSY, cdly20, FHW}).  Let $\alpha>0$ and $k$ be an integer such that $k>\alpha$. For $f\in\mathcal M$, the space $\mathcal L_{\alpha,k}^{c}(\mathcal T)$  is defined as a subset of $\mathcal M$ with finite norm given by
\begin{equation}\label{09}
 \left\|f\right\|_{\mathcal L_{\alpha,k}^{c}(\mathcal T)}=\|f\|_{\infty}+\sup_{t>0}\frac{1}{t^{\alpha}}  \left\|T_{t}|(I-T_{t})^{k}f|^{2}\right\|_{\infty}^{\frac{1}{2}}.
 \end{equation}
The spaces $\mathcal L_{\alpha,k}^{r}(\mathcal T)$ and $\mathcal L_{\alpha,k}^{cr}(\mathcal T)$ are defined in a standard way.

\begin{theorem}\label{26}
Let $\mathcal T$ be a 2-positive and contractive analytic semigroup and $\alpha>0$. Let $k$ be any integer greater than $\alpha$. Then the spaces
$\mathcal L_{\alpha,k}^{c}(\mathcal T)$ and $\mathcal L_{\alpha}^{c}(\mathcal T)$ coincide with equivalent norms.
\end{theorem}

Our approach is purely algebraic and does not require any further condition such as kernel estimates or Markovian property on the semigroup. This would admit rich applications whether or not the underlying von Neumann algebra is commutative.  We invite the reader to compare our algebraic method appearing in Section 3 with the classical arguments depending on kernel estimates and markovian property appearing in Appendix A.

\begin{remark}
{\rm As elaborated in the previous article (see \cite[Remark 1.5]{HJ24}), this paper remains focused on the study of inhomogeneous Campanato/Lipschitz spaces because no satisfactory noncommutative analogue of local $L_p$-spaces or tempered distributions exists to support a rigorous definition of the homogeneous versions. Crucially, every proof relies only on the homogeneous part of the norm--independent of the $\|\cdot\|_\infty$ term--so that, should a suitable noncommutative ``local" framework be developed in the future, all results herein extend directly and without substantive modification to the homogeneous setting. }
\end{remark}

\smallskip
The layout of the paper is as follows: Section 2 contains some basic assumptions about the semigroups of operators under consideration. Section 3 is devoted to the proof of Theorem \ref{e25}. In Section 4 we prove the self-improving property---Theorem \ref{26}. Finally, in the appendix, we collect all the characterizations of classical Campanato/Lipschitz spaces closely related to the ones in the present paper, and in particular give a `classical' proof of the equivalence with the spaces defined purely via Poisson semigroup for $0<\alpha<1/2$.


Throughout, we will use the notation $A\lesssim B$ to denote that $A\leq CB$ for a positive constant $C$ which is independent of the main parameters, and the notation $A\lesssim_{\alpha}B$ to indicate $A\leq C_{\alpha}B$ for a positive constant $C_{\alpha}$ which depends only on the parameter $\alpha$. The equivalence $A\simeq_{\alpha}B$ will mean $A\lesssim_{\alpha}B$ and $B\lesssim_{\alpha}A$.

\section{Preliminary \label{s2}}

\subsection{ Completely positive and contractive semigroups}
A von Neumann algebra is a weak-$*$-closed $\ast$-subalgebra containing the unit  of $B(H)$ for some Hilbert space $H$. Let $\mathcal M$ be a von Neumann algebra, and the operator norm is denoted by $\|\cdot\|_\infty$.  We say a map $T$ from $\mathcal M$ to $\mathcal M$ is positive if $Tf\geq0$ for $f\geq0$. For $n\geq 1$, a map $T$ is called $n$-positive if $I_{n}\otimes T$ is positive on $M_{n}\otimes \mathcal{M}$; and $T$ is called completely positive if it is $n$-positive for all $n$. Here $M_{n}$ is the algebra of $n\times n$ complex matrices and $I_{n}$ is the identity operator on $M_{n}$.

A family of linear operators $\mathcal T=(T_{t})_{t>0}$ is said to be a semigroup if for any $s,t>0$,
$$T_{0}=id \quad\text{and}\quad T_{t}T_{s}=T_{t+s}.$$
In this article, a semigroup $\mathcal T=(T_{t})_{t>0}$ is called 2-positive and contractive if it satisfies the following properties:
\begin{itemize}
  \item [(i)] Each $T_{t}$ is a 2-positive map on $\mathcal M$ for any $t>0$;
  \item [(ii)] Each $T_{t}$ is contractive on $\mathcal M$, i.e., $\|T_{t}f\|_{\infty}\leq\|f\|_{\infty}$ for any $t>0$ and $f\in \mathcal M$;
  \item [(iii)] For any $f\in \mathcal{M}$, $T_{t}f$ converges to $f$ in the weak-$*$ topology of $\mathcal{M}$ as $t\rightarrow0^{+}$.
\end{itemize}

As usual such a semigroup $\mathcal T=(T_{t})_{t>0}$ always admits an infinitesimal generator
\begin{equation*}
  A=\lim_{t\rightarrow0}\frac{id-T_{t}}{t}.
\end{equation*}
Then we may write $T_{t}=e^{-tA}$ and its weak-$*$ domain in $\mathcal{M}$ is defined as
\begin{equation*}
 \mathrm{dom}_{\infty}(A)=\left\{f\in \mathcal{M}: \,\,\lim_{t\rightarrow0}\frac{f-T_{t}f}{t}\,\, \text{exists}\right\},
\end{equation*}
where the limit is taken in the sense of the weak-$*$ topology of $\mathcal{M}$. By \cite[Lemma 1.1]{ebd80} and \cite[P694]{mj12}, one knows that $\mathrm{dom}_{\infty}(A)$ is weak-$*$ dense in  $\mathcal{M}$.


%
Given a semigroup $\mathcal T=(e^{-tA})_{t> 0}$, its subordinated Poisson semigroup $\mathcal P=(P_{t})_{t> 0}$ is defined by $P_{t}=e^{-t\sqrt{A}}$, which has the following identity (see e.g. \cite[Page 47]{s02}),
\begin{equation*}
  P_{t}=\frac{1}{2\sqrt{\pi}}\int_{0}^{\infty}t e^{-\frac{t^{2}}{4u}}u^{-\frac{3}{2}}T_{u}du.
\end{equation*}
\par
It follows from \cite[Corollary 2.8]{ch74} that the Kadison-Schwarz inequality holds for every unital 2-positive map on $\mathcal M$. The same argument shows that the unitality assumption may be replaced by the contractivity. More precisely, if $T:\mathcal M\to\mathcal M$ is a $2$-positive and contractive map, then for all $f\in\mathcal M$,
\begin{equation}\label{1}
|T(f)|^{2} \leq T(|f|^{2}).
\end{equation}

\subsection{Bounded analytic semigroups}
Because of \eqref{1}, we need to consider contractive semigroups in the present paper; however, all the statements in the present subsection hold for any bounded semigroups.
Let us introduce some definitions and simple facts about sectorial operators and analyticity. For any $\omega \in (0,\pi)$, let
$\Sigma_{\omega}=\{z\in \mathbb{C}\setminus\{0\}: |Arg(z)|<\omega\}$
be the open sector of angle $2\omega$ around the half-line $(0,\infty).$ We denote by $\overline{\Sigma_\omega}$ its closure in $\mathbb C$,
which contains the origin. Let $\sigma(A)$ be the spectrum of $A$ and for any $z\in \mathbb{C}\setminus\sigma(A)$, $R(z,A)$ be the corresponding resolvent operator. Recall that the operator $A$ is called a sectorial operator of type $\omega$ if  $A$ is closed, weak-$*$ densely defined, $\sigma(A)\subset\overline{\Sigma_{\omega}}$, and for any $\theta\in (\omega, \pi)$ there exists a constant $K_{\theta}>0$ such that
\begin{equation}\label{lkls}
\|zR(z,A)\|\leq K_{\theta},\quad z\in \mathbb{C}\setminus\overline{\Sigma_{\theta}}.
\end{equation}

Let $-A$ be the negative infinitesimal generator of a bounded weak-$*$ continuous semigroup $(T_t)_{t>0}$ on $\mathcal{M}$.
It is easy to check that $A$ is a sectorial operator of type $\frac{\pi}{2}$; moreover, for $0<\beta<1$, $A^{\beta}$ is a sectorial operator of type $\frac{\beta\pi}{2}$, see e.g. \cite[Proposition 2.5]{xu24}.

Now we introduce the notion of analytic semigroups on $\mathcal M$; similar definitions over general locally convex spaces can be found in \cite[Chapter 9]{kh64}.


\begin{definition}\label{sddddd}
Let $(T_t)_{t>0}$ be a bounded weak-$*$ continuous semigroup on $\mathcal{M}$. Then $(T_t)_{t>0}$ is called a bounded analytic semigroup if there exists a positive angle $\delta \in (0, \pi / 2)$ such that $(T_t)_{t > 0}$ can be extended to a bounded map $z \in \Sigma_\delta \mapsto T_z \in B(\mathcal M)$ satisfying
$$
\forall x \in \mathcal{M},\ \forall \varphi \in L^{1}(\mathcal{M}),\quad
z \mapsto \langle T_z x, \varphi \rangle
\quad\text{is holomorphic on }\Sigma_\delta.$$
\end{definition}

The following result, by a verbatim argument in \cite[Chapter 1.5]{GJA85}, establishes the connection between analyticity and sectoriality.

\begin{lemma}\label{ms}
The semigroup of operators $(e^{-tA})_{t>0}$ is bounded and analytic if and only if there exists  $\omega\in(0,\frac{\pi}{2})$ such that $A$ is a sectorial operator of type $\omega$.
\end{lemma}

Let $H^{\infty}(\Sigma_{\theta})$ denote the space of bounded analytic functions $\varphi: \Sigma_{\theta}\rightarrow\mathbb{C}$. This is a Banach algebra equipped with the norm
$$\|\varphi\|_{\infty,\theta}=\sup\{|\varphi(z)|: z\in\Sigma_{\theta}\}.$$
We then let $H^{\infty}_{0}(\Sigma_{\theta})$ be its subalgebra consisting of all $\varphi\in H^{\infty}(\Sigma_{\theta})$ satisfying
\begin{equation}\label{lklz}
|\varphi(z)|\lesssim\frac{|z|^{s}}{(1+|z|)^{2s}}
\end{equation}
for some positive number $s>0$. Let $\omega<\gamma<\theta<\pi$ and $\Gamma_{\gamma}$ be the boundary of $\Sigma_{\gamma}$ oriented counterclockwise such that
\[
\Gamma_{\gamma}(r)=
\begin{cases}
	-re^{i\gamma}, & r\in\mathbb{R}_{-},\\
	re^{-i\gamma}, & r\in\mathbb{R}_{+}.\\	
\end{cases}
\]
Then by \eqref{lkls} and \eqref{lklz}, for any $\varphi\in H^{\infty}_{0}(\Sigma_{\theta})$, the integral
$$\varphi(A)=\frac{1}{2\pi i}\int_{\Gamma_{\gamma}}\varphi(z)R(z,A)dz$$
is a well defined bounded operator on $D(A)$ and its weak-$*$ extension on $\mathcal M$ remains bounded.

The analyticity of the semigroup will be exploited in the present paper via the following proposition.

\begin{proposition}\label{wss}
Let $(e^{-tA})_{t>0}$ be a bounded weak-$*$ continuous semigroup on $\mathcal{M}$. Then for any $f\in \mathcal{M}$ and all integer $m\geq 1$, $(e^{-tA})_{t>0}$ is analytic if and only if
\begin{align}\label{wss0}\sup_{t>0}\left\| t^{m}\frac{\partial^{m}e^{-tA}f}{\partial t^{m}}\right\|_{\infty}\lesssim_{m}\left\| f\right\|_{\infty}.\end{align}
\end{proposition}
\begin{proof}
We only establish \eqref{wss0} for bounded analytic semigroups $(e^{-tA})_{t>0}$, as the converse follows from a verbatim argument in \cite[Theorem 2.5.2]{Pazy83}.

Let $(e^{-tA})_{t>0}$ be a bounded analytic semigroup. From Lemma \ref{ms}, one knows that $A$ is a sectorial operator of type $\omega\in(0,\frac{\pi}{2})$. Let $\theta\in(\omega, \frac{\pi}{2})$.
Fix $t>0$ and an integer $m\geq1$. Choosing $\varphi(z)=z^{m}e^{-tz}$, it is easy to check that $\varphi$ is a bounded analytic function which belongs to $H^{\infty}_{0}(\Sigma_{\theta})$. Thus, setting $\omega<\gamma<\theta<\frac{\pi}{2}$, by the holomorphic functional calculus, one gets that the integral
$$\varphi(A)=\frac{1}{2\pi i}\int_{\Gamma_{\gamma}}z^{m}e^{-tz}R(z,A)dz$$
is in the weak-$*$ topology and is a bounded operator on $\mathcal M$, where $\Gamma_{\gamma}$ is the boundary of $\Sigma_{\gamma}$ oriented counterclockwise. Thus
\begin{equation*}
\|\varphi(A)\|_{\infty}=\left\|\frac{1}{2\pi i}\int_{\Gamma_{\gamma}}\varphi(z)R(z,A)dz\right\|_{\infty}
 \leq \frac{1}{2\pi}\int_{\Gamma_{\gamma}}\left\|zR(z,A)\right\|_{\infty}|z^{m}e^{-tz}|\frac{|dz|}{|z|}.
\end{equation*}
Letting $z=re^{\pm i\gamma}\in\Gamma_{\gamma}$, we then have
\begin{equation*}
  \|\varphi(A)\|_{\infty}\leq \frac{2 K_{\theta}}{2\pi}\int_{0}^{\infty}r^{m-1}e^{-tr\cos(\gamma)}dr=\frac{K_{\theta}}{\pi (t\cos(\gamma))^{m}}\int_{0}^{\infty}r^{m-1}e^{-r}dr=\frac{(m-1)!K_{\theta}}{\pi (t\cos(\gamma))^{m}}.
\end{equation*}
This implies that for $f\in \mathcal M$,
$$\|\frac{\partial^{m}e^{-tA}f}{\partial t^{m}}\|_{\infty}=\|A^{m}e^{-tA}f\|_{\infty}=\|\varphi(A)f\|_{\infty}\leq \frac{(m-1)!K_{\theta}}{\pi (t\cos(\gamma))^{m}}\|f\|_{\infty}.$$
Taking the supremum over $t>0$ yields the desired assertion.
The proof is complete.
\end{proof}

\begin{remark}\label{jiajia}
{\rm From the previously mentioned fact that $A^{\frac{\beta}{2}}$ is a sectorial operator of type $\frac{\beta\pi}{2}$ for any $0<\beta<1$, together with Lemma \ref{ms}, the subordinated semigroup $(e^{-tA^{\frac{\beta}{2}}})_{t>0}$ is analytic. Consequently, \eqref{wss0} holds true for any bounded subordinated semigroup. In the previous work \cite[Proposition 2.3]{HJ24}, a strong version of \eqref{wss0} for subordinated Poisson semigroup was exploited.}
\end{remark}

\section{The proof of Theorem \ref{e25}}

In this section, we will prove the isomorphisms between semigroup Campanato spaces and semigroup Lipschitz spaces. Some basic properties, such as the completeness of these spaces, have been verified in the previous work \cite{HJ24}. To simplify notation, we will denote the homogeneous parts of the Campanato and Lipschitz norms respectively by
$$\|f\|_{\dot{\mathcal L}_{\alpha}^{c}(\mathcal T)}=\sup_{t>0}\frac{1}{t^{\alpha}}\left\|T_{t}|(I-T_{t})^{[\alpha]+1}f|^{2}\right\|_{\infty}^{\frac{1}{2}},
$$
and
$$\|f\|_{\dot{\Lambda}_{\alpha}({\mathcal T})}=\sup_{t>0}\frac{1}{t^{\alpha-([ \alpha]+1)}}\left\|\frac{\partial^{[ \alpha]+1}T_{t}f}{\partial t^{[\alpha]+1}}\right\|_{\infty}.$$

The following lemma is one of the novel ingredients allowing to deal with all $\alpha>0$.

\begin{lemma}\label{190}
Let $\mathcal T$ be a 2-positive and contractive analytic semigroup. Then we have for $\alpha>0$ and $f\in \mathcal M$,
  \begin{equation*}
 \left\|f\right\|_{\dot{\Lambda}_{\alpha}(\mathcal T)}\lesssim_{\alpha}\sup_{t>0}\frac{1}{t^{\alpha}}\left\|T_{\frac{t}{2}}\left|\big(I-T_{2^{[\frac{[ \alpha]+1}{[ \alpha]+1-\alpha}]+1}t}\big)^{[ \alpha]+1}f\right|^{2}\right\|_{\infty}^{\frac{1}{2}}.
\end{equation*}
\end{lemma}
\begin{proof}
Fix $\alpha>0$. We first claim that
\begin{equation}\label{ja}
 \left\|f\right\|_{\dot{\Lambda}_{\alpha}(\mathcal T)}\leq \frac{1}{1-\sum_{r=1}^{[ \alpha]+1}\frac{C_{[ \alpha]+1}^{r}}{(2^{[\frac{[ \alpha]+1}{[ \alpha]+1-\alpha}]+1}r+1)^{[ \alpha]+1-\alpha}}}\sup_{t>0}\frac{1}{t^{\alpha-([ \alpha]+1)}}\left\|\frac{\partial^{[ \alpha]+1} T_{t}}{\partial t^{[ \alpha]+1}}
\big(I-T_{2^{[\frac{[ \alpha]+1}{[ \alpha]+1-\alpha}]+1}t}\big)^{[ \alpha]+1}f\right\|_{\infty}.
\end{equation}
Once this is done, by Proposition \ref{wss} and the Kadison-Schwarz inequality \eqref{1}, one has for $t>0$ fixed,
\begin{equation*}
\begin{aligned}
 &\frac{1}{t^{\alpha-([ \alpha]+1)}}\left\|\frac{\partial^{[ \alpha]+1} T_{t}}{\partial t^{[ \alpha]+1}}
(I-T_{2^{[\frac{[ \alpha]+1}{[ \alpha]+1-\alpha}]+1}t})^{[ \alpha]+1}f\right\|_{\infty}
\\=&\frac{2^{[ \alpha]+1}}{t^{\alpha}}\left\|\Big(s^{[ \alpha]+1}\frac{\partial^{[ \alpha]+1} T_{s}}{\partial s^{[ \alpha]+1}}\Big)\Big|_{s=\frac t2}\left(T_{\frac{t}{2}}(I-T_{2^{[\frac{[ \alpha]+1}{[ \alpha]+1-\alpha}]+1}t})^{[\alpha]+1}f\right)\right\|_{\infty}
\\ \lesssim&\frac{2^{[ \alpha]+1}}{t^{\alpha}}\left\|T_{\frac{t}{2}}(I-T_{2^{[\frac{[ \alpha]+1}{[ \alpha]+1-\alpha}]+1}t})^{[\alpha]+1}f\right\|_{\infty}
\\=& \frac{2^{[ \alpha]+1}}{t^{\alpha}}\left\|\left|T_{\frac{t}{2}}\big(I-T_{2^{[\frac{[ \alpha]+1}{[ \alpha]+1-\alpha}]+1}t}\big)^{[\alpha]+1}f\right|^{2}\right\|_{\infty}^{\frac{1}{2}}
\\ \leq& \frac{2^{[ \alpha]+1}}{t^{\alpha}}\left\|T_{\frac{t}{2}}\left|\big(I-T_{2^{[\frac{[ \alpha]+1}{[ \alpha]+1-\alpha}]+1}t}\big)^{[\alpha]+1}f\right|^{2}\right\|_{\infty}^{\frac{1}{2}}.
\end{aligned}
\end{equation*}
Then taking supremum over $t>0$ on both sides of the above inequality,  \eqref{ja} yields the desired assertion.

Now we are left to deal with \eqref{ja}.
  Note that
\begin{equation*}
\begin{aligned}
  I=&(I-T_{2^{[\frac{[ \alpha]+1}{[ \alpha]+1-\alpha}]+1}t})^{[ \alpha]+1}+\Big[I-(I-T_{2^{[\frac{[ \alpha]+1}{[ \alpha]+1-\alpha}]+1}t})^{[ \alpha]+1}\Big]
\\=&(I-T_{2^{[\frac{[ \alpha]+1}{[ \alpha]+1-\alpha}]+1}t})^{[ \alpha]+1}+\sum_{r=1}^{[ \alpha]+1}(-1)^{r+1}C_{[ \alpha]+1}^{r}T_{2^{[\frac{[ \alpha]+1}{[ \alpha]+1-\alpha}]+1}rt}.
\end{aligned}
\end{equation*}
Applying this identity, one gets for a fixed $t>0$,
\begin{equation}\label{cccc}
\begin{aligned}
  &\frac{1}{t^{\alpha-([ \alpha]+1)}}\left\|\frac{\partial^{[ \alpha]+1} T_{t}f}{\partial t^{[ \alpha]+1}}\right\|_{\infty}
\\=&\frac{1}{t^{\alpha-([ \alpha]+1)}}\left\|\frac{\partial^{[ \alpha]+1} T_{t}}{\partial t^{[ \alpha]+1}}
\left[\big(I-T_{2^{[\frac{[ \alpha]+1}{[ \alpha]+1-\alpha}]+1}t}\big)^{[ \alpha]+1}+\sum_{r=1}^{[ \alpha]+1}(-1)^{r+1}C_{[ \alpha]+1}^{r}T_{2^{[\frac{[ \alpha]+1}{[ \alpha]+1-\alpha}]+1}rt}\right]f\right\|_{\infty}
\\ \leq&\frac{1}{t^{\alpha-([ \alpha]+1)}}\left\|\frac{\partial^{[ \alpha]+1} T_{t}}{\partial t^{[ \alpha]+1}}
\big(I-T_{2^{[\frac{[ \alpha]+1}{[ \alpha]+1-\alpha}]+1}t}\big)^{[ \alpha]+1}f\right\|_{\infty}
+\sum_{r=1}^{[ \alpha]+1}C_{[ \alpha]+1}^{r}\frac{1}{t^{\alpha-([ \alpha]+1)}}\left\|\frac{\partial^{[ \alpha]+1} T_{t}}{\partial t^{[ \alpha]+1}}T_{2^{[\frac{[ \alpha]+1}{[ \alpha]+1-\alpha}]+1}rt}f\right\|_{\infty}
\\ =&\frac{1}{t^{\alpha-([ \alpha]+1)}}\left\|\frac{\partial^{[ \alpha]+1} T_{t}}{\partial t^{[ \alpha]+1}}
\big(I-T_{2^{[\frac{[ \alpha]+1}{[ \alpha]+1-\alpha}]+1}t}\big)^{[ \alpha]+1}f\right\|_{\infty}
\\&+\sum_{r=1}^{[ \alpha]+1}\frac{C_{[ \alpha]+1}^{r}}{\big(2^{[\frac{[ \alpha]+1}{[ \alpha]+1-\alpha}]+1}r+1\big)^{[ \alpha]+1-\alpha}}\frac{1}{\left((2^{[\frac{[ \alpha]+1}{[ \alpha]+1-\alpha}]+1}r+1)t\right)^{\alpha-([ \alpha]+1)}}\left\|\left.\frac{\partial^{[ \alpha]+1} T_{s}f}{\partial s^{[\alpha]+1}}\right|_{s=(2^{[\frac{[ \alpha]+1}{[ \alpha]+1-\alpha}]+1}r+1)t}\right\|_{\infty}
\\ \leq&\sup_{t>0}\frac{1}{t^{\alpha-([ \alpha]+1)}}\left\|\frac{\partial^{[ \alpha]+1} T_{t}}{\partial t^{[ \alpha]+1}}
\big(I-T_{2^{[\frac{[ \alpha]+1}{[ \alpha]+1-\alpha}]+1}t}\big)^{[ \alpha]+1}f\right\|_{\infty}
+\sum_{r=1}^{[ \alpha]+1}\frac{C_{[ \alpha]+1}^{r}}{(2^{[\frac{[ \alpha]+1}{[ \alpha]+1-\alpha}]+1}r+1)^{[ \alpha]+1-\alpha}}\left\|f\right\|_{\Lambda_{\alpha}(\mathcal T)}.
\end{aligned}
\end{equation}
Note that
 $$\sum_{r=1}^{[ \alpha]+1}C_{[ \alpha]+1}^{r}=\sum_{r=0}^{[ \alpha]+1}C_{[ \alpha]+1}^{r}-C_{[ \alpha]+1}^{0}=2^{[ \alpha]+1}-1$$ and for $r\geq1$,
\begin{equation*}
\begin{aligned}
\frac{1}{(2^{[\frac{[ \alpha]+1}{[ \alpha]+1-\alpha}]+1}r+1)^{[ \alpha]+1-\alpha}}
<\frac{1}{(2^{\frac{[ \alpha]+1}{[ \alpha]+1-\alpha}}r)^{[ \alpha]+1-\alpha}} =\frac{1}{2^{[ \alpha]+1}r^{[ \alpha]+1-\alpha}}\leq\frac{1}{2^{[ \alpha]+1}}.
\end{aligned}
\end{equation*}
Thus it is easy to see
\begin{equation}\label{la}
 \sum_{r=1}^{[ \alpha]+1}\frac{C_{[ \alpha]+1}^{r}}{(2^{[\frac{[ \alpha]+1}{[ \alpha]+1-\alpha}]+1}r+1)^{[ \alpha]+1-\alpha}}
 <\sum_{r=1}^{[ \alpha]+1}\frac{C_{[ \alpha]+1}^{r}}{2^{[ \alpha]+1}}=\frac{2^{[ \alpha]+1}-1}{2^{[ \alpha]+1}}<1.
\end{equation}
 Then taking the supremum over $t>0$ on the left-hand side of \eqref{cccc} and using \eqref{la}, one obtains
\begin{equation*}
\begin{aligned}
 \left\|f\right\|_{\dot{\Lambda}_{\alpha}(\mathcal T)}\leq \frac{1}{1-\sum_{r=1}^{[ \alpha]+1}\frac{C_{[ \alpha]+1}^{r}}{(2^{[\frac{[ \alpha]+1}{[ \alpha]+1-\alpha}]+1}r+1)^{[ \alpha]+1-\alpha}}}\sup_{t>0}\frac{1}{t^{\alpha-([ \alpha]+1)}}\left\|\frac{\partial^{[ \alpha]+1} T_{t}}{\partial t^{[ \alpha]+1}}
\big(I-T_{2^{[\frac{[ \alpha]+1}{[ \alpha]+1-\alpha}]+1}t}\big)^{[ \alpha]+1}f\right\|_{\infty}.
\end{aligned}
\end{equation*}
\end{proof}

\begin{remark}\label{lip1}
{\rm From the above arguments, in particular by noticing \eqref{la}, one may see that the index $2^{[\frac{[ \alpha]+1}{[ \alpha]+1-\alpha}]+1}$ in Lemma \ref{190} can be replaced by any integer $b$ such that
$$\sum_{r=1}^{[ \alpha]+1}\frac{C_{[ \alpha]+1}^{r}}{(br+1)^{[ \alpha]+1-\alpha}}<1.$$}
\end{remark}

The following identity is the second ingredient to help concluding the desired result for all $\alpha>0$.

\begin{lemma}\label{531}
For every $t>0$ and any integers $M,Q\geq1$, we have
$$(I-T_{Qt})^{M}=(I-T_{\frac{t}{2}})^{M}\left(I+\overline{\sum}
T_{\sum_{i=1}^{2Q-1}\frac{r_{i}t}{2}}\right),$$
where $\overline{\sum}$ is the weighted multiple summation
 $$\overline{\sum}=\sum_{r_{1}=1}^{M}\sum_{r_{2}=0}^{r_{1}}
 \cdots\sum_{r_{2Q-1}=0}^{r_{2Q-2}}C_{M}^{r_{1}}C_{r_{1}}^{r_{2}}\cdots C_{r_{2Q-1}}^{r_{2Q-2}}$$ with $0\leq r_{2Q-1} \leq r_{2Q-2}\cdots\leq r_{2}\leq r_{1}\leq M$ and $C^{r}_{M}=\frac{M!}{r!(M-r)!}$.
\end{lemma}

\begin{proof}
Since the identity $(T_{s})^{a}=T_{as}$ holds trivially for any $s>0$ and positive integer $a$ by using the semigroup property, for fixed $t>0$ we deduce
\begin{equation*}
\begin{aligned}
  (I-T_{Qt})^{M}=&(I-T_{\frac{t}{2}}+T_{\frac{t}{2}}-T_{t}+\cdots+T_{(Q-\frac{1}{2})t}-T_{Qt})^{M}
\\=&(I-T_{\frac{t}{2}})^{M}\big(I+T_{\frac{t}{2}}+T_{t}+\cdots+T_{(Q-\frac{1}{2})t}\big)^{M}
\\=&(I-T_{\frac{t}{2}})^{M}(I+\sum_{r_{1}=1}^{M}C_{M}^{r_{1}}(T_{\frac{t}{2}}+T_{t}+\cdots+T_{(Q-\frac{1}{2})t})^{r_{1}})
\\=&(I-T_{\frac{t}{2}})^{M}\big(I+\sum_{r_{1}=1}^{M}C_{M}^{r_{1}}\sum_{r_{2}=0}^{r_{1}}C_{r_{1}}^{r_{2}}T_{\frac{(r_{1}-r_{2})t}{2}}(T_{t}+\cdots+T_{(Q-\frac{1}{2})t})^{r_{2}}\big)
\\=&(I-T_{\frac{t}{2}})^{M}\left(I+\sum_{r_{1}=1}^{M}C_{M}^{r_{1}}\sum_{r_{2}=0}^{r_{1}}C_{r_{1}}^{r_{2}}
 T_{\frac{(r_{1}-r_{2})t}{2}}\sum_{r_{3}=0}^{r_{2}}C_{r_{2}}^{r_{3}}
 T_{(r_{2}-r_{3})t}\right.
 \\&\left.\cdots\sum_{r_{2Q-1}=0}^{r_{2Q-2}}C_{r_{2Q-1}}^{r_{2Q-2}}T_{(Q-1)(r_{2Q-2}-r_{2Q-1})t}T_{(Q-\frac{1}{2})r_{2Q-1}t}\right)
\\=&(I-T_{\frac{t}{2}})^{M}\left(I+\sum_{r_{1}=1}^{M}\sum_{r_{2}=0}^{r_{1}}
 \cdots\sum_{r_{2Q-1}=0}^{r_{2Q-2}}C_{M}^{r_{1}}C_{r_{1}}^{r_{2}}\cdots C_{r_{2Q-1}}^{r_{2Q-2}}
T_{\sum_{i=1}^{2Q-1}\frac{r_{i}t}{2}}\right).
\end{aligned}
\end{equation*}
Therefore, we obtain the desired result by denoting $$\overline{\sum}=\sum_{r_{1}=1}^{M}\sum_{r_{2}=0}^{r_{1}}
 \cdots\sum_{r_{2Q-1}=0}^{r_{2Q-2}}C_{M}^{r_{1}}C_{r_{1}}^{r_{2}}\cdots C_{r_{2Q-1}}^{r_{2Q-2}}.$$
\end{proof}

We now come to prove Theorem \ref{e25}.

\textbf{The proof of Theorem \ref{e25}.}
Let $\alpha>0$ and $f\in \mathcal M$. By definitions, it suffices to show  $$\left\|f\right\|_{\dot{\mathcal L}_{\alpha}^{c}(\mathcal T)}\simeq_{\alpha}\left\|f\right\|_{\dot{\Lambda}_{\alpha}(\mathcal T)}.$$
We first show
\begin{equation}\label{e1}
 \left\|f\right\|_{\dot{\mathcal L}_{\alpha}^{c}(\mathcal T)}\lesssim_{\alpha}\left\|f\right\|_{\dot{\Lambda}_{\alpha}(\mathcal T)}.
\end{equation}

Fix $t>0$. Using \cite[Lemma 2.1]{FHW} and Applying Lemma \ref{531} to $M=[\alpha]+1$ and $Q=1$, we obtain
\begin{equation}\label{150}
\begin{aligned}
 \frac{1}{t^{\alpha}}\left\|T_{t}\left|(I-T_{t})^{[\alpha]+1}f\right|^{2}\right\|_{\infty}^{\frac{1}{2}}= & \frac{1}{t^{\alpha}}\left\|T_{t}\left|(I-T_{\frac{t}{2}})^{[\alpha]+1}(I+\sum_{N=1}^{[\alpha]+1}C_{[\alpha]+1}^{N}T_{\frac{Nt}{2}})f\right|^{2}\right\|_{\infty}^{\frac{1}{2}}
\\ \leq&\frac{1}{t^{\alpha}}\left\|T_{t}|(I-T_{\frac{t}{2}})^{[\alpha]+1}f|^{2}\right\|_{\infty}^{\frac{1}{2}}
\\  & +\frac{1}{t^{\alpha}}\left\|T_{t}\left|\sum_{N=1}^{[\alpha]+1}C_{[\alpha]+1}^{N}T_{\frac{Nt}{2}}(I-T_{\frac{t}{2}})^{[\alpha]+1}f\right|^{2}\right\|_{\infty}^{\frac{1}{2}}
\\ =&:A+B.
\end{aligned}
\end{equation}
By the contraction of $T_t$, one gets
\begin{equation*}
\begin{aligned}
 A\leq&\sup_{t>0}\frac{1}{t^{\alpha}}\left\|T_{t}|(I-T_{\frac{t}{2}})^{[\alpha]+1}f|^{2}\right\|_{\infty}^{\frac{1}{2}}
 \\=&\sup_{t>0}\frac{1}{(2t)^{\alpha}}\left\|T_{2t}|(I-T_{t})^{[\alpha]+1}f|^{2}\right\|_{\infty}^{\frac{1}{2}}
\\ \leq&\sup_{t>0}\frac{1}{(2t)^{\alpha}}\left\|T_{t}|(I-T_{t})^{[\alpha]+1}f|^{2}\right\|_{\infty}^{\frac{1}{2}}
\\=&\frac{1}{2^{\alpha}}\left\|f\right\|_{\dot{\mathcal L}_{\alpha}^{c}(\mathcal T)}.
\end{aligned}
\end{equation*}
For the term $B$, one has
\begin{equation*}
\begin{aligned}
  B\leq & \frac{1}{t^{\alpha}}\left\|\sum_{N=1}^{[\alpha]+1}C_{[\alpha]+1}^{N}T_{\frac{Nt}{2}}(I-T_{\frac{t}{2}})^{[\alpha]+1}f\right\|_{\infty}
\\ \leq&\sum_{N=1}^{[\alpha]+1}C_{[\alpha]+1}^{N}\frac{1}{t^{\alpha}}\left\|T_{\frac{Nt}{2}}(I-T_{\frac{t}{2}})^{[\alpha]+1}f\right\|_{\infty}
\\=&\sum_{N=1}^{[\alpha]+1}C_{[\alpha]+1}^{N}\frac{1}{t^{\alpha}}\left\|\int_{0}^{\frac{t}{2}}\frac{\partial T_{s_{1}}}{\partial s_{1}}T_{\frac{Nt}{2}}(I-T_{\frac{t}{2}})^{[\alpha]}fds_{1}\right\|_{\infty}
\\=&\sum_{N=1}^{[\alpha]+1}C_{[\alpha]+1}^{N}\frac{1}{t^{\alpha}}\left\|\int_{0}^{\frac{t}{2}}\frac{\partial T_{s_{1}}}{\partial s_{1}}\cdots\int_{0}^{\frac{t}{2}}\frac{\partial T_{s_{[\alpha]}}}{\partial s_{[\alpha]}} \int_{\frac{Nt}{2}}^{\frac{(N+1)t}{2}}\frac{\partial T_{s_{[\alpha]+1}}f}{\partial s_{[\alpha]+1}}ds_{[\alpha]+1}ds_{[\alpha]}\cdots ds_{1}\right\|_{\infty}.
\end{aligned}
\end{equation*}
Note that
 $$\frac{\partial T_{s}}{\partial s}\frac{\partial T_{t}}{\partial t}=\frac{\partial^{2} T_{t+s}}{\partial s^{2}}=\frac{\partial^{2} T_{t+s}}{\partial t^{2}}.$$
Applying this property and letting $v=s_{1}+\cdots+s_{[\alpha]+1}$, we have
\begin{equation*}
\begin{aligned}
  B\leq&\sum_{N=1}^{[\alpha]+1}C_{[\alpha]+1}^{N}\frac{1}{t^{\alpha}}\left\|\int_{0}^{\frac{t}{2}}\cdots\int_{0}^{\frac{t}{2}}\int_{\frac{Nt}{2}}^{\frac{(N+1)t}{2}}\frac{\partial^{[\alpha]+1} T_{s_{1}+s_{2}+\cdots s_{[\alpha]+1}}f}{\partial s_{[\alpha]+1}^{[\alpha]+1}}ds_{[\alpha]+1}ds_{[\alpha]}\cdots ds_{1}\right\|_{\infty}
   \\  \nonumber=&\sum_{N=1}^{[\alpha]+1}C_{[\alpha]+1}^{N}\frac{1}{t^{\alpha}}\left\|\int_{0}^{\frac{t}{2}}\cdots\int_{0}^{\frac{t}{2}} \int_{\frac{Nt}{2}+s_{1}+\cdots+s_{[\alpha]}}^{\frac{(N+1)t}{2}+s_{1}+\cdots+s_{[\alpha]}}\frac{\partial^{[\alpha]+1} T_{v}f}{\partial v^{[\alpha]+1}}dvds_{[\alpha]}\cdots ds_{1}\right\|_{\infty}
   \\ \nonumber\leq&\sum_{N=1}^{[\alpha]+1}C_{[\alpha]+1}^{N}\frac{1}{t^{\alpha}}\int_{0}^{\frac{t}{2}}\cdots\int_{0}^{\frac{t}{2}} \int_{\frac{Nt}{2}+s_{1}+\cdots+s_{[\alpha]}}^{\frac{(N+1)t}{2}+s_{1}+\cdots+s_{[\alpha]}}\frac{v^{\alpha-([\alpha]+1)}}{v^{\alpha-([\alpha]+1)}}\left\|\frac{\partial^{[\alpha]+1} T_{v}f}{\partial v^{[\alpha]+1}}\right\|_{\infty}dvds_{[\alpha]}\cdots ds_{1}
   \\ \nonumber \leq&\left\|f\right\|_{\dot{\Lambda}_{\alpha}(\mathcal T)}\sum_{N=1}^{[\alpha]+1}C_{[\alpha]+1}^{N}\frac{1}{t^{\alpha}}\int_{0}^{\frac{t}{2}}\cdots\int_{0}^{\frac{t}{2}} \int_{\frac{Nt}{2}+s_{1}+\cdots+s_{[\alpha]}}^{\frac{(N+1)t}{2}+s_{1}+\cdots+s_{[\alpha]}}{v^{\alpha-([\alpha]+1)}}dvds_{[\alpha]}\cdots ds_{1}.
\end{aligned}
\end{equation*}
If $\alpha$ is not an integer, one gets
\begin{equation}\label{bhi}
\begin{aligned}
  &\int_{0}^{\frac{t}{2}}\cdots\int_{0}^{\frac{t}{2}} \int_{\frac{Nt}{2}+s_{1}+\cdots+s_{[\alpha]}}^{\frac{(N+1)t}{2}+s_{1}+\cdots+s_{[\alpha]}}{v^{\alpha-([\alpha]+1)}}dvds_{[\alpha]}\cdots ds_{1}
\\ \leq&\int_{0}^{\frac{t}{2}}\cdots\int_{0}^{\frac{t}{2}} \frac{(\frac{(N+1)t}{2}+s_{1}+\cdots+s_{[\alpha]})^{\alpha-[\alpha]}}{\alpha-[\alpha]}ds_{[\alpha]}\cdots ds_{1}
\\ \leq&\frac{\big([\alpha]+N+1\big)^{\alpha-[\alpha]}t^{\alpha}}{2^{\alpha}(\alpha-[\alpha])}.
\end{aligned}
\end{equation}
If $\alpha$ is an integer, we have
\begin{equation}\label{162}
\begin{aligned}
 &\int_{0}^{\frac{t}{2}}\cdots\int_{0}^{\frac{t}{2}} \int_{\frac{Nt}{2}+s_{1}+\cdots+s_{[\alpha]}}^{\frac{(N+1)t}{2}+s_{1}+\cdots+s_{[\alpha]}}{v^{\alpha-([\alpha]+1)}}dvds_{[\alpha]}\cdots ds_{1}
  \\ =&\int_{0}^{\frac{t}{2}}\cdots\int_{0}^{\frac{t}{2}}  \int_{\frac{Nt}{2}+s_{1}+\cdots+s_{[\alpha]}}^{\frac{(N+1)t}{2}+s_{1}+\cdots+s_{[\alpha]}}{v^{-1}}dvds_{\alpha}\cdots ds_{1}
  \\ \leq &\int_{0}^{\frac{t}{2}}\cdots\int_{0}^{\frac{t}{2}} \frac{\frac{t}{2}} {\frac{Nt}{2}+s_{1}+\cdots+s_{\alpha}}ds_{\alpha}\cdots ds_{1} \leq \frac{t^{\alpha}}{2^{\alpha}N}.
\end{aligned}
\end{equation}
Thus, by \eqref{bhi} and \eqref{162}, we obtain
\begin{equation*}
  B\leq C_\alpha\left\|f\right\|_{\dot{\Lambda}_{\alpha}(\mathcal T)}
\end{equation*}
for some constant $C_\alpha>0$. Taking the supremum over $t>0$ on the left-hand side of \eqref{150} and combining the estimates of the terms $A$ and $B$, we have
$$\|f\|_{\dot{\mathcal L}_{\alpha}^{c}(\mathcal T)}\leq \frac{1}{2^\alpha}\|f\|_{\dot{\mathcal L}_{\alpha}^{c}(\mathcal T)}+C_\alpha\|f\|_{\dot{\Lambda}_{\alpha}(\mathcal T)}.$$
 Then by the fact that $\frac{1}{2^{\alpha}}<1$, one gets
$$\left\|f\right\|_{\dot{\mathcal L}_{\alpha}^{c}(\mathcal T)}\lesssim_{\alpha}\left\|f\right\|_{\dot{\Lambda}_{\alpha}(\mathcal T)}.$$

Now we prove turn to the converse inequality of \eqref{e1},
\begin{equation*}
  \left\|f\right\|_{\dot{\Lambda}_{\alpha}(\mathcal T)}\lesssim_{\alpha} \left\|f\right\|_{\dot{\mathcal L}_{\alpha}^{c}(\mathcal T)}.
\end{equation*}
By Lemma \ref{190}, it suffices to show
\begin{equation}\label{e2}
\sup_{t>0}\frac{1}{t^{\alpha}}\left\|T_{\frac{t}{2}}\left|
(I-T_{bt})^{[\alpha]+1}f\right|^{2}\right\|_{\infty}^{\frac{1}{2}}\lesssim \left\|f\right\|_{\dot{\mathcal L}_{\alpha}^{c}(\mathcal T)},
\end{equation}
where $b=2^{[\frac{[ \alpha]+1}{[ \alpha]+1-\alpha}]+1}$. Let $t>0$. Applying Lemma \ref{531} to $M=[ \alpha]+1$ and $Q=b$ and by the Kadison-Schwarz inequality \eqref{1}, one gets
\begin{equation*}
\begin{aligned}
\frac{1}{t^{\alpha}}\left\|T_{\frac{t}{2}}\left|(I-T_{bt})^{[\alpha]+1}f\right|^{2}\right\|_{\infty}^{\frac{1}{2}}
 \leq&\frac{1}{t^{\alpha}}\left\| T_{\frac{t}{2}}\left|
(I-T_{t})^{[\alpha]+1}(I+\overline{\sum}
T_{\sum_{i=1}^{2b-1}\frac{r_{i}t}{2}})f\right|^{2}\right\|_{\infty}^{\frac{1}{2}}
\\ \leq&\frac{1}{t^{\alpha}}\left\| T_{\frac{t}{2}}\left|
(I-T_{\frac{t}{2}})^{[\alpha]+1}f\right|^{2}\right\|_{\infty}^{\frac{1}{2}}
+\frac{1}{t^{\alpha}}\left\| T_{\frac{t}{2}}\left|
(I-T_{\frac{t}{2}})^{[\alpha]+1}\overline{\sum}
T_{\sum_{i=1}^{2b-1}\frac{r_{i}t}{2}}f\right|^{2}\right\|_{\infty}^{\frac{1}{2}}
\\ \leq& \frac{1}{2^{\alpha}}\left\|f\right\|_{\dot{\mathcal L}_{\alpha}^{c}(\mathcal T)}+\overline{\sum}\frac{1}{t^{\alpha}}\left\| T_{\sum_{i=1}^{2b-1}\frac{r_{i}t}{2}+\frac t2}\left|
(I-T_{\frac{t}{2}})^{[\alpha]+1}f\right|^{2}\right\|_{\infty}^{\frac{1}{2}}
\\ \leq& \frac{1}{2^{\alpha}}\left\|f\right\|_{\dot{\mathcal L}_{\alpha}^{c}(\mathcal T)}+\overline{\sum}\sup_{t>0}\frac{1}{t^{\alpha}}\left\| T_{t}\left|
(I-T_{t})^{[\alpha]+1}f\right|^{2}\right\|_{\infty}^{\frac{1}{2}}
\\ \leq& \big(\frac{1}{2^{\alpha}}+\overline{\sum}1\big)\left\|f\right\|_{\dot{\mathcal L}_{\alpha}^{c}(\mathcal T)}.
\end{aligned}
\end{equation*}
Taking the supremum over $t>0$ on the left-hand side of the above inequality yields \eqref{e2}. The proof of Theorem \ref{e25} is complete.
\hfill$\Box$\\

\section{The proof of Theorem \ref{26}}
In this part, we are going to show that Campanato space $\mathcal L_{\alpha}^{c}(\mathcal T)$ admits a self-improving property for all $\alpha>0$, which allows every element in this space has a higher-order cancellation property as in the case of BMO spaces \cite{cdly20,DSY,FHW}. Let $\alpha>0$ and $k$ be an integer greater than $\alpha$. As in the last section, the homogeneous part will be denoted by
$$\left\|f\right\|_{\dot{\mathcal L}_{\alpha,k}^{c}(\mathcal T)}=\sup_{t>0}\frac{1}{t^{\alpha}} \left\|T_{t}|(I-T_{t})^{k}f|^{2}\right\|_{\infty}^{\frac{1}{2}}.$$
The self-improving property of Campanato space will be established via the similar property of Lipschitz spaces. Let $\alpha>0$ and $k\geq [\alpha]+2$ be an integer. Define
\begin{equation*}
  \Lambda_{\alpha,k}({\mathcal T})=\{f\in \mathcal M: \left\|f\right\|_{\Lambda_{\alpha,k}({\mathcal T})}<\infty\},
\end{equation*}
where
\begin{equation*}
\left\|f\right\|_{\Lambda_{\alpha,k}({\mathcal T})}=\|f\|_{\infty}+ \left\|f\right\|_{\dot{\Lambda}_{\alpha,k}({\mathcal T})},\,\,\, \left\|f\right\|_{\dot{\Lambda}_{\alpha,k}({\mathcal T})}=\sup_{t>0}\frac{1}{t^{ \alpha-k}}\left\|\frac{\partial^{k}T_{t}f}{\partial t^{k}}\right\|_{\infty}.
\end{equation*}

Again, we will only need to deal with the equivalence of the homogeneous norms.
The following result has been obtained by the first two authors in \cite[Proposition 3.2]{HJ24} when the underlying von Neumann algebra is finite. Here we will provide a slightly different approach in general cases.

\begin{proposition}\label{20}
Let $\mathcal T=(T_t)_{t>0}$ be a 2-positive and contractive analytic semigroup acting on a von Neumann algebra $\mathcal M$. Let $\alpha>0$ and let $k$ be an integer such that $k\geq[\alpha]+2$.
Then, for every $f\in\mathcal M$, the following statements hold:
\begin{itemize}
  \item [(i)]$\left\|f\right\|_{\dot{\Lambda}_{\alpha,k}({\mathcal T})}\lesssim2^{(k-[\alpha]-1)(\frac{k+[\alpha]+2}{2}-\alpha)}\left\|f\right\|_{\dot{\Lambda}_{\alpha}({\mathcal T})};$
  \item [(ii)]$\left\|f\right\|_{\dot{\Lambda}_{\alpha}({\mathcal T})}\leq\max\{1,\frac{1}{([\alpha]+1-\alpha)([\alpha]+2-\alpha)\cdots(k-\alpha)}\}\left\|f\right\|_{\dot{\Lambda}_{\alpha,k}({\mathcal T})}$.
\end{itemize}
\end{proposition}
\begin{proof}
(i) Let $\ell\geq[\alpha]+2$. For $t>0$ fixed, we use Proposition \ref{wss} to obtain
\begin{equation*}
\begin{aligned}
\frac{1}{t^{\alpha-\ell}}\left\|\frac{\partial^{\ell}T_{t}f}{\partial t^{\ell}}\right\|_{\infty}
=&\frac{1}{t^{\alpha-\ell}}\left\|\left.\frac{\partial T_{v}}{\partial v}\right|_{v=\frac{t}{2}}\left.\frac{\partial^{\ell-1}T_{s}f}{\partial s^{\ell-1}}\right|_{s=\frac{t}{2}}\right\|_{\infty}
 \\ \nonumber =&2^{\ell-\alpha}\frac{1}{(\frac{t}{2})^{\alpha-(\ell-1)}}\left\|\left.\frac{\partial^{\ell-1}T_{s}f}{\partial s^{\ell-1}}\right|_{s=\frac{t}{2}}\right\|_{\infty}
\\ \nonumber \leq&2^{\ell-\alpha}\sup_{t>0}\frac{1}{t^{\alpha-(\ell-1)}}\left\|\frac{\partial^{\ell-1}T_{t}f}{\partial t^{\ell-1}}\right\|_{\infty}.
\end{aligned}
\end{equation*}
Taking the supremum over $t>0$ on the left-hand side, one gets
\begin{equation}\label{dnl}
 \left\|f\right\|_{\dot{\Lambda}_{\alpha,\ell}({\mathcal T})}\lesssim2^{\ell-\alpha} \left\|f\right\|_{\dot{\Lambda}_{\alpha,\ell-1}({\mathcal T})}.
\end{equation}
 Then applying (\ref{dnl}) to $\ell=k,\dotsm,[\alpha]+2$, we obtain the desired assertion by iteration.

(ii) We first claim
\begin{equation}\label{as}
  \left\|f\right\|_{\dot{\Lambda}_{\alpha,d}({\mathcal T})}\leq \frac{1}{1-2^{\alpha-d}}\sup_{t>0}\frac{1}{t^{\alpha-d}}\left\|\frac{\partial^{d}T_{t}}{\partial t^{d}}(f-T_{t}f)\right\|_{\infty}
\end{equation}
for any integer $d\geq[\alpha]+1$. Indeed, for $d\geq[\alpha]+1$ and $t>0$, one has
\begin{equation*}
\begin{aligned}
  \frac{1}{t^{\alpha-d}}\left\|\frac{\partial^{d}T_{t}}{\partial t^{d}}f\right\|_{\infty}=&\frac{1}{t^{\alpha-d}}\left\|\frac{\partial^{d}T_{t}}{\partial t^{d}}(f-T_{t}f+T_{t}f)\right\|_{\infty}
  \\ \leq& \frac{1}{t^{\alpha-d}}\left\|\frac{\partial^{d}T_{t}}{\partial t^{d}}(f-T_{t}f)\right\|_{\infty}+\frac{1}{t^{\alpha-d}}\left\|\frac{\partial^{d}T_{t}}{\partial t^{d}}T_{t}f\right\|_{\infty}
\\=&\frac{1}{t^{\alpha-d}}\left\|\frac{\partial^{d}T_{t}}{\partial t^{d}}(f-T_{t}f)\right\|_{\infty}+2^{\alpha-d}\frac{1}{(2t)^{\alpha-d}}\left\|\left.\frac{\partial^{d}T_{s}f}{\partial s^{d}}\right|_{s=2t}\right\|_{\infty}
\\ \leq&\sup_{t>0}\frac{1}{t^{\alpha-d}}\left\|\frac{\partial^{d}T_{t}}{\partial t^{d}}(f-T_{t}f)\right\|_{\infty}+2^{\alpha-d} \left\|f\right\|_{\dot{\Lambda}_{\alpha,d}({\mathcal T})}.
\end{aligned}
\end{equation*}
Then one can easily get \eqref{as} by taking the supremum over $t>0$ on the left-hand side and using the fact that $2^{\alpha-d} <1$ since $d>\alpha$.

Now we prove the desired assertion. Let $\ell\geq[\alpha]+2$ and $t>0$. Note that
\begin{equation}\label{r5}
\begin{aligned}
\frac{1}{t^{\alpha-(\ell-1)}}
\left\|
\frac{\partial^{\ell-1}T_t}{\partial t^{\ell-1}}
(f-T_tf)
\right\|_{\infty}
=&
\frac{1}{t^{\alpha-(\ell-1)}}
\left\|
\frac{\partial^{\ell-1}T_t}{\partial t^{\ell-1}}
\left(
-\int_0^t \frac{\partial T_s f}{\partial s}\,ds
\right)
\right\|_{\infty}
 \\ =&\frac{1}{t^{\alpha-(\ell-1)}}\left\|\int_{0}^{t}\frac{\partial^{\ell}T_{t+s}f}{\partial s^{\ell}}ds\right\|_{\infty}
 \\ (v=t+s)=&\frac{1}{t^{\alpha-(\ell-1)}}\left\|\int_{t}^{2t}\frac{\partial^{\ell}T_{v}f}{\partial v^{\ell}}dv\right\|_{\infty}
 \\ \leq&\frac{1}{t^{\alpha-(\ell-1)}}\int_{t}^{2t}\frac{v^{\alpha-\ell}}{v^{\alpha-\ell}}\left\|\frac{\partial^{\ell}T_{v}f}{\partial v^{\ell}}\right\|_{\infty}dv
\\ \leq&\sup_{v>0}\left\{\frac{1}{v^{\alpha-\ell}}\left\|\frac{\partial^{\ell}T_{v}f}{\partial v^{\ell}}\right\|_{\infty}\right\}\frac{1}{t^{\alpha-(\ell-1)}}\int_{t}^{2t}v^{\alpha-\ell}dv
\\ =&\frac{1-2^{\alpha-(\ell-1)}}{\ell-1-\alpha}\left\|f\right\|_{\dot{\Lambda}_{\alpha,\ell}({\mathcal T})}.
\end{aligned}
\end{equation}
Taking the supremum over $t>0$ on the left-hand side, it is easy to see
$$\frac{1}{1-2^{\alpha-(\ell-1)}}\sup_{t>0}\frac{1}{t^{\alpha-(\ell-1)}}\left\|\frac{\partial^{(\ell-1)}T_{t}}{\partial t^{(\ell-1)}}(f-T_{t}f)\right\|_{\infty}\leq\frac{1}{\ell-1-\alpha}\left\|f\right\|_{\dot{\Lambda}_{\alpha,\ell}({\mathcal T})}.$$
By the claim \eqref{as}, one gets
\begin{equation*}
 \left\|f\right\|_{\dot{\Lambda}_{\alpha,\ell-1}({\mathcal T})}
  \leq\frac{1}{\ell-1-\alpha}\left\|f\right\|_{\dot{\Lambda}_{\alpha,\ell}({\mathcal T})}.
\end{equation*}
By iteration, that is, applying \eqref{r5} to $\ell=[\alpha]+2,\dotsm,k$, one deduces
\begin{equation*}
  \left\|f\right\|_{\dot{\Lambda}_{\alpha}({\mathcal T})}\leq\frac{1}{([\alpha]+1-\alpha)([\alpha]+2-\alpha)\cdots(k-\alpha)}\left\|f\right\|_{\dot{\Lambda}_{\alpha,k}({\mathcal T})}.
\end{equation*}
\end{proof}

The following lemma is the higher-order analogue of Lemma \ref{190}.

\begin{lemma}\label{wd}
Let $\mathcal T=(T_t)_{t>0}$ be a 2-positive and contractive analytic semigroup acting on a von Neumann algebra
$\mathcal M$. Let $\alpha>0$ and let $k$ be an integer such that $k>\alpha$.
Then, for every $f\in\mathcal M$, we have
  \begin{equation*}
 \left\|f\right\|_{\dot{\Lambda}_{\alpha,k}(\mathcal T)}\lesssim_{\alpha,k}\sup_{t>0}\frac{1}{t^{\alpha}}\left\|T_{\frac{t}{2}}\left|\big(I-T_{2^{[\frac{k}{k-\alpha}]+1}t}\big)^{k}f\right|^{2}\right\|_{\infty}^{\frac{1}{2}}.
\end{equation*}
\end{lemma}

\begin{proof}
The proof follows the same argument as that of Lemma \ref{190}. We only spell out the absorption step, since the corresponding
coefficient changes in the higher-order setting.

Fix $\alpha>0$ and an integer $k>\alpha$. By the identity
\[
I=
\left(I-T_{2^{[\frac{k}{k-\alpha}]+1}t}\right)^k
+
\sum_{r=1}^{k}(-1)^{r+1}C_k^r
T_{2^{[\frac{k}{k-\alpha}]+1}rt},
\]
together with the semigroup property, we obtain
\begin{equation*}
 \frac{1}{t^{\alpha-k}}
\left\|
\frac{\partial^kT_t f}{\partial t^k}
\right\|_\infty
\leq
\frac{1}{t^{\alpha-k}}
\left\|
\frac{\partial^kT_t}{\partial t^k}
\left(I-T_{2^{[\frac{k}{k-\alpha}]+1}t}\right)^k f
\right\|_\infty
+
\sum_{r=1}^{k}
\frac{C_k^r}
{\left(2^{[\frac{k}{k-\alpha}]+1}r+1\right)^{k-\alpha}}
\|f\|_{\dot{\Lambda}_{\alpha,k}(\mathcal T)}.
\end{equation*}
Moreover,
\[
\sum_{r=1}^{k}
\frac{C_k^r}
{\left(2^{[\frac{k}{k-\alpha}]+1}r+1\right)^{k-\alpha}}
<
\frac{1}{2^k}\sum_{r=1}^{k}C_k^r
=
\frac{2^k-1}{2^k}
<1,
\]
since
\[
2^{\left([\frac{k}{k-\alpha}]+1\right)(k-\alpha)}>2^k.
\]
Hence, taking the supremum over $t>0$ and absorbing the last term,
we get
\[
\|f\|_{\dot{\Lambda}_{\alpha,k}(\mathcal T)}
\leq
\frac{1}{
1-\displaystyle\sum_{r=1}^{k}
\frac{C_k^r}
{\left(2^{[\frac{k}{k-\alpha}]+1}r+1\right)^{k-\alpha}}
}
\sup_{t>0}
\frac{1}{t^{\alpha-k}}
\left\|
\frac{\partial^kT_t}{\partial t^k}
\left(I-T_{2^{[\frac{k}{k-\alpha}]+1}t}\right)^k f
\right\|_\infty.
\]

It remains to estimate the last term. As in the proof of
Lemma \ref{190}, Proposition \ref{wss} and the Kadison--Schwarz inequality
\eqref{1} yield
\begin{equation*}
  \frac{1}{t^{\alpha-k}}
\left\|
\frac{\partial^kT_t}{\partial t^k}
\left(I-T_{2^{[\frac{k}{k-\alpha}]+1}t}\right)^k f
\right\|_\infty
\lesssim_k
\frac{1}{t^\alpha}
\left\|
T_{t/2}
\left|
\left(I-T_{2^{[\frac{k}{k-\alpha}]+1}t}\right)^k f
\right|^2
\right\|_\infty^{1/2}.
\end{equation*}
Taking the supremum over $t>0$ gives the desired estimate.
\end{proof}

The following theorem can be regarded as a strengthened version of Theorem \ref{e25}.
\begin{theorem}\label{261}
 Let $\alpha>0$ and let $k$ be an integer such that $k\geq[\alpha]+2$. If $\mathcal T=(T_t)_{t>0}$ is a 2-positive and contractive analytic semigroup on a von Neumann algebra $\mathcal M$,
then the space $\mathcal L^c_{\alpha,k}(\mathcal T)$ is isomorphic to $\Lambda_{\alpha,k}(\mathcal T)$ with equivalent norms.
\end{theorem}

\begin{proof}
We first show
\begin{equation}\label{w4}
 \|f\|_{\dot{\mathcal L}_{\alpha,k}^{c}(\mathcal T)}\lesssim_{\alpha,k}\|f\|_{\dot{\Lambda}_{\alpha,k}(\mathcal T)}.
\end{equation}

We start with the case of $0<\alpha<1$. Let $t>0$ be fixed. One may compute
\begin{equation*}
\begin{aligned}
  \frac{1}{t^{\alpha}}\left\|T_{t}|(I-T_{t})^{k}f|^{2}\right\|_{\infty}^{\frac{1}{2}} \leq&\frac{1}{t^{\alpha}}\left\|(I-T_{t})^{k}f\right\|_{\infty}
\\=&\frac{1}{t^{\alpha}}\left\|\int^{t}_{0}\cdots\int^{t}_{0}\int^{t}_{0}\frac{\partial^{k}T_{s_{1}+\cdots +s_{k-1}+s_{k}}f}{\partial s_{k}^{k}}ds_{k}ds_{k-1}\cdots ds_{1}\right\|_{\infty}.
\end{aligned}
\end{equation*}
Let $v=s_{1}+\cdots +s_{k-1}+s_{k}$. Then by the fact that
$$\alpha-k<\alpha+1-k<\cdots<\alpha+(k-2)-k<-1,\;-1<\alpha-1<0,$$
it is easy to see
\begin{equation*}
\begin{aligned}
  \frac{1}{t^{\alpha}}\left\|T_{t}|(I-T_{t})^{k}f|^{2}\right\|_{\infty}^{\frac{1}{2}}\leq&
\frac{1}{t^{\alpha}}\left\|\int^{t}_{0}\cdots\int^{t}_{0}\int^{t+s_{1}+\cdots +s_{k-1}}_{s_{1}+\cdots +s_{k-1}}\frac{\partial^{k}T_{v}f}{\partial v^{k}}dvds_{k-1}\cdots ds_{1}\right\|_{\infty}
\\ \leq&\frac{1}{t^{\alpha}}\int^{t}_{0}\cdots\int^{t}_{0}\int^{t+s_{1}+\cdots +s_{k-1}}_{s_{1}+\cdots +s_{k-1}}\frac{v^{\alpha-k}}{v^{\alpha-k}}\left\|\frac{\partial^{k}T_{v}f}{\partial v^{k}}\right\|_{\infty}dvds_{k-1}\cdots ds_{1}
\\ \leq&\left\|f\right\|_{\dot{\Lambda}_{\alpha,k}(\mathcal T)}\frac{1}{t^{\alpha}}\int^{t}_{0}\cdots\int^{t}_{0}\int^{t+s_{1}+\cdots +s_{k-1}}_{s_{1}+\cdots +s_{k-1}}v^{\alpha-k}dvds_{k-1}\cdots ds_{1}
\\ \leq&\left\|f\right\|_{\dot{\Lambda}_{\alpha,k}(\mathcal P)}\frac{1}{t^{\alpha}}\int^{t}_{0}\frac{s_{1}^{\alpha-1}}{(k-(\alpha+1))\cdots(1-\alpha)} ds_{1}
\\ =&\frac{1}{(k-(\alpha+1))\cdots(1-\alpha)\alpha}\left\|f\right\|_{\dot{\Lambda}_{\alpha,k}(\mathcal T)}.
\end{aligned}
\end{equation*}
Taking supremum over $t>0$ on the left-hand side, one gets
\begin{equation}\label{xv}
 \left\|f\right\|_{\dot{\mathcal L}_{\alpha,k}^{c}(\mathcal T)}\leq\frac{1}{(k-(\alpha+1))(k-(\alpha+2))\cdots(1-\alpha)\alpha}\left\|f\right\|_{\dot{\Lambda}_{\alpha,k}(\mathcal T)},
\end{equation}

Now let us consider the case where $\alpha\geq1$. By Lemma \ref{531}, for $t>0$ fixed, one obtains
\begin{equation*}
\begin{aligned}
 \frac{1}{t^{\alpha}}\left\|T_{t}|(I-T_{t})^{k}f|^{2}\right\|_{\infty}^{\frac{1}{2}}
=&  \frac{1}{t^{\alpha}}\left\|T_{t}\left|(I-T_{\frac{t}{2}})^{k}(I+\sum_{r=1}^{k}C_{k}^{r}T_{\frac{r}{2}t})f\right|^{2}\right\|_{\infty}^{\frac{1}{2}}
\\ \nonumber
   \leq&\frac{1}{t^{\alpha}}\left\|T_{t}\left|(I-T_{\frac{t}{2}})^{k}f\right|^{2}\right\|_{\infty}^{\frac{1}{2}}
+\sum_{r=1}^{k}C_{k}^{r}\frac{1}{t^{\alpha}}
  \left\|T_{t}\left|T_{\frac{rt}{2}}(I-T_{\frac{t}{2}})^{k}f\right|^{2}\right\|_{\infty}^{\frac{1}{2}}
\\ \nonumber \leq&
\frac{1}{2^{\alpha}}\left\|f\right\|_{\dot{\mathcal L}_{\alpha,k}^{c}(\mathcal T)}+\sum_{r=1}^{k}C_{k}^{r}\frac{1}{t^{\alpha}}
  \left\|T_{t}\left|T_{\frac{rt}{2}}(I-T_{\frac{t}{2}})^{k}f\right|^{2}\right\|_{\infty}^{\frac{1}{2}}.
\end{aligned}
\end{equation*}
On the other hand, note that
\begin{equation*}
\begin{aligned}
&\sum_{r=1}^{k}C_{k}^{r}\frac{1}{t^{\alpha}}
  \left\|T_{\frac{rt}{2}}(I-T_{\frac{t}{2}})^{k}f\right\|_{\infty}
  \\=&\sum_{r=1}^{k}C_{k}^{r}\frac{1}{t^{\alpha}}\left\|\int^{\frac{t}{2}}_{0}\cdots\int^{\frac{t}{2}}_{0}\int^{\frac{(r+1)t}{2}}_{\frac{rt}{2}}\frac{\partial^{k}T_{s_{1}+\cdots +s_{k-1}+s_{k}}f}{\partial s_{k}^{k}}ds_{k}ds_{k-1}\cdots ds_{1}\right\|_{\infty}.
\end{aligned}
\end{equation*}
Thus, letting $s=s_{1}+\cdots +s_{k-1}+s_{k}$ and using the condition that $\alpha+1-k<0$, one gets
\begin{equation*}
\begin{aligned}
&\sum_{r=1}^{k}C_{k}^{r}\frac{1}{t^{\alpha}}
  \left\|T_{t}\left|T_{\frac{rt}{2}}(I-T_{\frac{t}{2}})^{k}f\right|^{2}\right\|_{\infty}^{\frac{1}{2}}\\ \leq&\sum_{r=1}^{k}C_{k}^{r}\frac{1}{t^{\alpha}}\left\|\int^{\frac{t}{2}}_{0}\cdots\int^{\frac{t}{2}}_{0}\int^{\frac{(r+1)t}{2}+s_{1}+\cdots +s_{k-1}}_{\frac{rt}{2}+s_{1}+\cdots +s_{k-1}}\frac{\partial^{k}T_{s}f}{\partial s^{k}}dsds_{k-1}\cdots ds_{1}\right\|_{\infty}
 \\ \leq&\sum_{r=1}^{k}C_{k}^{r}\frac{1}{t^{\alpha}}\int^{\frac{t}{2}}_{0}\cdots\int^{\frac{t}{2}}_{0}\int^{\frac{(r+1)t}{2}+s_{1}+\cdots +s_{k-1}}_{\frac{rt}{2}+s_{1}+\cdots +s_{k-1}}\frac{s^{\alpha-k}}{s^{\alpha-k}}\left\|\frac{\partial^{k}T_{s}f}{\partial s^{k}}\right\|_{\infty}dsds_{k-1}\cdots ds_{1}
  \\ \leq&\left\|f\right\|_{\dot{\Lambda}_{\alpha,k}(\mathcal T)}\sum_{r=1}^{k}C_{k}^{r}\frac{1}{t^{\alpha}}\int^{\frac{t}{2}}_{0}\cdots\int^{\frac{t}{2}}_{0}\frac{(\frac{rt}{2}+s_{1}+\cdots +s_{k-1})^{\alpha+1-k}}{k-(\alpha+1)}ds_{k-1}\cdots ds_{1}
  \\ \leq&\sum_{r=1}^{k}C_{k}^{r}\frac{r^{\alpha+1-k}}{2^{\alpha}(k-(\alpha+1))}\left\|f\right\|_{\dot{\Lambda}_{\alpha,k}(\mathcal T)}.
\end{aligned}
\end{equation*}
This gives
\begin{equation*}
 \frac{1}{t^{\alpha}}\left\|T_{t}|(I-T_{t})^{k}f|^{2}\right\|_{\infty}^{\frac{1}{2}}\leq
\frac{1}{2^{\alpha}}\left\|f\right\|_{\dot{\mathcal L}_{\alpha,k}^{c}(\mathcal T)}+\sum_{r=1}^{k}C_{k}^{r}\frac{r^{\alpha+1-k}}{2^{\alpha}(k-(\alpha+1))}\left\|f\right\|_{\dot{\Lambda}_{\alpha,k}(\mathcal T)}.
\end{equation*}
Taking the supremum over $t>0$ on the left-hand side, one obtains
\begin{equation*}
  \left\|f\right\|_{\dot{\mathcal L}_{\alpha,k}^{c}(\mathcal T)}\leq\frac{1}{2^{\alpha}}\left\|f\right\|_{\dot{\mathcal L}_{\alpha,k}^{c}(\mathcal T)}+\sum_{r=1}^{k}C_{k}^{r}\frac{r^{\alpha+1-k}}{2^{\alpha}(k-(\alpha+1))}\left\|f\right\|_{\dot{\Lambda}_{\alpha,k}(\mathcal T)}.
\end{equation*}
Since $\frac{1}{2^{\alpha}}<1$ holds for any $\alpha\geq1$, one then deduces
\begin{equation}\label{xc}\left\|f\right\|_{\dot{\mathcal L}_{\alpha,k}^{c}(\mathcal T)}\leq\frac{1}{2^{\alpha}-1}\sum_{r=1}^{k}C_{k}^{r}\frac{r^{\alpha+1-k}}{(k-(\alpha+1))}\|f\|_{\dot{\Lambda}_{\alpha,k}(\mathcal T)}.\end{equation}
Therefore, combining \eqref{xv} with \eqref{xc}, we get \eqref{w4}.

\smallskip

Now we prove the converse direction
\begin{equation*}
     \|f\|_{\dot{\Lambda}_{\alpha,k}(\mathcal T)}\lesssim_{\alpha,k} \|f\|_{\dot{\mathcal L}_{\alpha,k}^{c}(\mathcal T)}.
\end{equation*}
To this end, by Lemma \ref{wd}, we only need to show
\begin{equation}\label{w6}
\sup_{t>0}\frac{1}{t^{\alpha}}\left\| T_{\frac{t}{2}}\left|
\big(I-T_{2^{[\frac{k}{k-\alpha}]+1}t}\big)^{k}f\right|^{2}\right\|_{\infty}^{\frac{1}{2}}\lesssim_{\alpha,k} \left\|f\right\|_{\dot{\mathcal L}_{\alpha,k}^{c}(\mathcal T)}.
\end{equation}
Applying Lemma \ref{531} with $M=k$ and $Q=2^{[\frac{k}{k-\alpha}]+1}$, and using the Kadison-Schwarz inequality \eqref{1}, one gets
\begin{equation*}
\begin{aligned}
 &\frac{1}{t^{\alpha}}\left\| T_{\frac{t}{2}}\left|
(I-T_{2^{[\frac{k}{k-\alpha}]+1}t})^{k}f\right|^{2}\right\|_{\infty}^{\frac{1}{2}}
\\ \leq&\frac{1}{t^{\alpha}}\left\|T_{\frac{t}{2}}\left|
(I-T_{\frac{t}{2}})^{k}f\right|^{2}\right\|_{\infty}^{\frac{1}{2}}
+\overline{\sum}\frac{1}{t^{\alpha}}\left\| T_{\frac{t}{2}}\left|
(I-T_{\frac{t}{2}})^{k}
T_{\sum_{i=1}^{2^{[\frac{k}{k-\alpha}]+2}-1}\frac{r_{i}}{2}t}f\right|^{2}\right\|_{\infty}^{\frac{1}{2}}
\\ \leq &\frac{1}{2^{\alpha}}\left\|f\right\|_{\dot{\mathcal L}_{\alpha}^{c}(\mathcal T)}
+\overline{\sum}\frac{1}{t^{\alpha}}\left\| T_{\sum_{i=1}^{2^{[\frac{k}{k-\alpha}]+2}-1}\frac{r_{i}t}{2}+\frac t2}\left|
(I-T_{\frac{t}{2}})^{k}f\right|^{2}\right\|_{\infty}^{\frac{1}{2}}
\\ \leq &\frac{1}{2^{\alpha}}\left\|f\right\|_{\dot{\mathcal L}_{\alpha}^{c}(\mathcal T)}
+\overline{\sum}\sup_{t>0}\frac{1}{t^{\alpha}}\left\| T_{t}\left|
(I-T_{t})^{k}f\right|^{2}\right\|_{\infty}^{\frac{1}{2}}
\\ \leq &\big(\frac{1}{2^{\alpha}}+\overline{\sum}1\big)\left\|f\right\|_{\dot{\mathcal L}_{\alpha,k}^{c}(\mathcal T)}.
\end{aligned}
\end{equation*}
Taking the supremum over $t>0$ on the left-hand side of the above inequality yields the assertion \eqref{w6}. Therefore, this proof is complete.
\end{proof}

Now one may conclude Theorem \ref{26}.

\textbf{The proof of Theorem \ref{26}.}
Let $f\in \mathcal M$. Combining the results in Theorem \ref{e25}, Proposition \ref{20} and Theorem \ref{261}, it is easy to get the following equivalences
\begin{equation*}
 \left\|f\right\|_{\mathcal L_{\alpha}^{c}(\mathcal T)}\simeq_{\alpha}\left\|f\right\|_{\Lambda_{\alpha}(\mathcal T)}\simeq_{\alpha,k}\left\|f\right\|_{\Lambda_{\alpha,k}(\mathcal T)}\simeq_{\alpha,k}\left\|f\right\|_{\mathcal L_{\alpha,k}^{c}(\mathcal T)}.
\end{equation*}
\hfill$\Box$\\

As Corollary \ref{jy} follows from Theorem \ref{e25}, one may deduce the following corollary from Theorem \ref{261}.
\begin{coro}
Suppose $\alpha>0$ and let $k$ be any integer greater than $\alpha$.
If, additionally, $(T_t)_{t>0}$ is completely positive, then the column space $\mathcal L_{\alpha,k}^{c}(\mathcal T)$ is completely isomorphic to the row space $\mathcal L_{\alpha,k}^{r}(\mathcal T)$ with equivalent norms.
\end{coro}

\appendix
\section{The case of the classical heat/Poisson semigroups} In this appendix, we collect all the characterizations of classical Campanato/Lipschitz spaces on $\mathbb R^n$ closely related to the ones defined by the heat/Poisson semigroups, and in particular we will give a `classical' proof of the equivalence with the spaces defined purely via Poisson semigroup for $0<\alpha<1/2$. Moreover, we establish the desired coincidences between the semigroup Campanato spaces and their little versions that were first introduced in \cite{HJ24}, see \eqref{little1} and \eqref{little2} below.

On $\mathbb R^n$. Let $\Delta=-\sum_{i=1}^{n}\frac{\partial^{2}}{\partial x_{i}^{2}}$ be the Laplacian. The heat semigroup $\mathcal H=(H_{t})_{t>0}$ where $H_{t}=e^{-t\Delta}$  admitting the kernel $$h_{t}(x)=\frac{\exp(-\frac{|x|^2}{4t})}{(4\pi t)^{\frac{n}{2}}};$$
the associated Poisson semigroup $\mathcal P=(P_{t})_{t>0}$ where $P_{t}=e^{-t\sqrt{\Delta}}$ having kernel
\begin{equation}\label{pkou}
 p_{t}(x)=\frac{c_{n}t}{(t^{2}+|x|^{2})^{\frac{n+1}{2}}},
\end{equation}
where $c_{n}=\frac{\Gamma(\frac{n+1}{2})}{\pi^{\frac{n+1}{2}}}$. Let $\alpha>0$. The spaces $\Lambda_{\alpha}(\mathbb{R}^{n})$, $\mathcal{L}_{\alpha,\Delta}(\mathbb{R}^{n})$, $\mathcal{L}_{\alpha,\sqrt{\Delta}}(\mathbb{R}^{n})$ are defined as subspaces of $L^{\infty}(\mathbb{R}^{n})$ with the following finite norms (see e.g. \cite{dy09}, \cite{v85} and \cite[p26]{gl09}):
\begin{equation}\label{bjiao}
  \|f\|_{\Lambda_{\alpha}(\mathbb{R}^{n})}=\|f\|_{\infty}+\sup_{|t|\neq0,x\in\mathbb{R}^{n}}\frac{\Big|D^{[\alpha]+1}_{t}f(x)\Big|}{|t|^{\alpha}},
\end{equation}
where $D_t^{[\alpha]+1}f(x)
=
\sum_{r=0}^{[\alpha]+1}
(-1)^{[\alpha]+1-r}
\binom{[\alpha]+1}{r}
f(x+rt)$;

$$\|f\|_{\mathcal{L}_{\alpha,\Delta}(\mathbb{R}^{n})}=\|f\|_{\infty}+\sup_{x\in\mathbb{R}^{n},t>0}\frac{1}{|B(x,\sqrt{t})|^{\frac{\alpha}{n}}}\left(\frac{1}{|B(x,\sqrt{t})|}\int_{B(x,\sqrt{t})}|(I-e^{-t\Delta})^{[\alpha]+1}f(y)|^{2}dy\right)^{\frac{1}{2}};$$

$$\|f\|_{\mathcal{L}_{\alpha,\sqrt{\Delta}}(\mathbb{R}^{n})}=\|f\|_{\infty}+\sup_{x\in\mathbb{R}^{n},t>0}\frac{1}{|B(x,t)|^{\frac{\alpha}{n}}}\left(\frac{1}{|B(x,t)|}\int_{B(x,t)}|(I-e^{-t\sqrt{\Delta}})^{[\alpha]+1}f(y)|^{2}dy\right)^{\frac{1}{2}}.$$
For convenience, the homogeneous parts in $\|\cdot\|_{\mathcal{L}_{\alpha}(\mathbb{R}^{n})}$, $\|\cdot\|_{{\Lambda}_{\alpha}(\mathbb{R}^{n})}$, $\|\cdot\|_{\mathcal{L}_{\alpha,\Delta}(\mathbb{R}^{n})}$ and $\|\cdot\|_{\mathcal{L}_{\alpha,\sqrt{\Delta}}(\mathbb{R}^{n})}$ will be denoted as $\|\cdot\|_{\dot{\mathcal{L}}_{\alpha}(\mathbb{R}^{n})}$, $\|\cdot\|_{\dot{\Lambda}_{\alpha}(\mathbb{R}^{n})}$, $\|\cdot\|_{\dot{\mathcal{L}}_{\alpha,\Delta}(\mathbb{R}^{n})}$ and $\|\cdot\|_{\dot{\mathcal{L}}_{\alpha,\sqrt{\Delta}}(\mathbb{R}^{n})}$, respectively.

\begin{theorem}\label{c1c}
{\rm (i)} For all $\alpha>0$, the spaces $\mathcal{L}_{\alpha}(\mathbb{R}^{n})$,  $\mathcal{L}_{\frac{\alpha}{2}}(\mathcal{H})$, $\mathcal{L}_{\alpha,\Delta}(\mathbb{R}^{n})$, $\Lambda_{\alpha}(\mathbb{R}^{n})$, $\Lambda_{\frac{\alpha}{2}}(\mathcal{H})$, $\Lambda_{\alpha}(\mathcal{P})$, and $\mathcal{L}_{{\alpha}}(\mathcal{P})$ coincide with equivalent norms.

{\rm (ii)} When $0<\alpha<1/2$, all the above spaces are isomorphic to $\mathcal{L}_{\alpha,\sqrt{\Delta}}(\mathbb{R}^{n})$ with equivalent norms.

\end{theorem}
\begin{proof}
(i) Let $\alpha>0$ and $f\in L_\infty(\mathbb R^n)$. As concluded from \cite[Theorem 3.4]{ddy05}, \cite[Proposition 2.4]{dy09} and our Theorem \ref{26}, one has
\begin{equation}\label{pt1}
\begin{aligned}
  \|f\|_{\mathcal{\dot{L}}_{\alpha}(\mathbb{R}^{n})}\simeq&\|f\|_{\mathcal{\dot{L}}_{\alpha,\Delta}(\mathbb{R}^{n})}
\\ \simeq& \sup_{t>0}\frac{1}{t^{\frac{\alpha}{2}}}\left\|e^{-t\Delta}\left|(I-e^{-t\Delta})^{[\alpha]+1}f\right|^{2}\right\|^{1/2}
\\  \simeq& \sup_{t>0}\frac{1}{t^{\frac{\alpha}{2}}}\left\|e^{-t\Delta}\left|(I-e^{-t\Delta})^{[\frac{\alpha}{2}]+1}f\right|^{2}\right\|^{1/2}
\\ \simeq&\|f\|_{\mathcal{\dot{L}}_{\frac{\alpha}{2}}(\mathcal{H})}.
\end{aligned}
\end{equation}
Secondly, one can derive from \cite{TW80}, \cite[Theorem 6.3.7]{gl09} and \cite[Corollary 3 (i)]{t82} that
\begin{equation}\label{pt2}
  \|f\|_{\mathcal{\dot{L}}_{\alpha}(\mathbb{R}^{n})}\simeq \|f\|_{\dot{\Lambda}_{\alpha}(\mathbb{R}^{n})}\simeq\|f\|_{\dot{\Lambda}_{\frac{\alpha}{2}}(\mathcal{H})}\simeq\|f\|_{\dot{\Lambda}_{\alpha}(\mathcal{P})}.
\end{equation}
Finally, our main result---Theorem \ref{e25} gives
\begin{equation}\label{pt3}
\|f\|_{\dot{\Lambda}_{\alpha}(\mathcal{P})}\simeq \|f\|_{\dot{\mathcal L}_{\alpha}(\mathcal{P})}.
\end{equation}
 Integrating \eqref{pt1}, \eqref{pt2}, and \eqref{pt3} leads to the desired assertion.
\par
(ii) This assertion follows from \eqref{pt3} and \cite[Proposition 2.4]{dy09}, where Duong and Yan obtained for $0<\alpha<1/2$,
\begin{equation}\label{dy}
\|\cdot\|_{\mathcal{\dot{L}}_{\alpha,\sqrt{\Delta}}(\mathbb{R}^{n})}\simeq\|\cdot\|_{\mathcal{\dot{L}}_{\alpha}(\mathcal{P})}.\end{equation}
\end{proof}

\begin{remark}\label{zz}
{\rm At the moment of writing, we have no idea how to extend Theorem \ref{c1c} (ii) or \eqref{dy} to all $\alpha\geq1/2$. Below, by making use of \eqref{pt2} and \eqref{dy}, we provide a classical proof of (ii) and \eqref{pt3} based on kernel estimates or properties of Poisson kernels. The proof will be finished by showing for $f\in L_\infty(\mathbb R^n)$,
\begin{equation}\label{dy1}
\|f\|_{\mathcal{\dot{L}}_{\alpha,\sqrt{\Delta}}(\mathbb{R}^{n})}\lesssim \|f\|_{\mathcal{\dot{L}}_{\alpha}(\mathbb{R}^{n})},
\end{equation}

\begin{equation}\label{dy2}
\|f\|_{\dot{\Lambda}_{\alpha}(\mathbb{R}^{n})}\lesssim\|f\|_{\mathcal{\dot{L}}_{\alpha}(\mathcal{P})}.
\end{equation}

Fix $f\in L_\infty(\mathbb R^n)$. We first show \eqref{dy1}. Let $B=B(x_{0},t_{B})$ be a fixed ball centered at $x_{0}$ and of radius $t_{B}$, and $f_{B}=\frac{1}{|B|}\int_{B}f(y)dy$. The Markovian property
\begin{equation}\label{lL}
  \int_{\mathbb{R}^{n}}p_{t}(y)dy=1
\end{equation}
 gives for any $x\in\mathbb R^n$
$$|f(x)-P_{t_{B}}f(x)|\leq |f(x)-f_{B}|+|P_{t_{B}}f(x)-f_{B}|=|f(x)-f_{B}|+|P_{t_{B}}(f-f_{B})(x)|.$$
Fix $x\in B$, we apply the H\"older inequality to get for $0<\alpha<1/2$,
\begin{equation*}
\begin{aligned}
 |P_{t_{B}}(f-f_{B})(x)|\leq & \int_{\mathbb{R}^{n}}\left|\frac{c_{n}t_{B}}{(t_{B}^{2}+|x-y|^{2})^{\frac{n+1}{2}}}\right||f(y)-f_{B}|dy  \\
  = & \int_{\mathbb{R}^{n}}\left(\frac{t_{B}}{(t_{B}^{2}+|x-y|^{2})^{\frac{n}{4}+\frac{1}{2}}}|f(y)-f_{B}|\right)
  \left(\frac{(t_{B}^{2}+|x-y|^{2})^{\frac{n}{4}+\frac{1}{2}}}{t_{B}}\left|\frac{c_{n}t_{B}}{(t_{B}^{2}+|x-y|^{2})^{\frac{n+1}{2}}}\right|\right)dy
\\ \leq&
\left(\int_{\mathbb{R}^{n}}\left|\frac{t_{B}^{2}}{(t_{B}^{2}+|x-y|^{2})^{\frac{n}{2}+1}}\right||f(y)-f_{B}|^{2}dy\right)^{\frac{1}{2}}
 \left(\int_{\mathbb{R}^{n}} \frac{1}{(t_{B}^{2}+|x-y|^{2})^{\frac{n}{2}}}dy\right)^{\frac{1}{2}}
\\ \lesssim&\left(\int_{\mathbb{R}^{n}}\left|\frac{t_{B}^{2}}{(t_{B}^{2}+|x-y|^{2})^{\frac{n}{2}+1}}\right||f(y)-f_{B}|^{2}dy\right)^{\frac{1}{2}}
\\ \leq&\left(\frac{1}{t_{B}^{n}}\int_{|x-y|\leq t_{B}}|f(y)-f_{B}|^{2}dy+ \sum_{k=1}^{\infty}2^{-2k}\frac{1}{(2^{k}t_{B})^{n}}\int_{2^{k-1}t_{B}\leq|x-y|\leq 2^{k}t_{B}}|f(y)-f_{B}|^{2}dy \right)^{\frac{1}{2}}
\\ \leq&\left(\frac{1}{t_{B}^{n}}\int_{|x-y|\leq t_{B}}|f(y)-f_{B}|^{2}dy+ \sum_{k=1}^{\infty}2^{-2k}\frac{1}{(2^{k}t_{B})^{n}}\int_{|x-y|\leq 2^{k}t_{B}}|f(y)-f_{B}|^{2}dy \right)^{\frac{1}{2}}
\\ \lesssim&\big(1+\sum_{k=1}^{\infty}2^{-2k(1-\alpha)}\big)t_{B}^{\alpha}\|f\|_{\dot{\mathcal{L}}_{\alpha}(\mathbb{R}^{n})}.
\end{aligned}
\end{equation*}
Therefore
\begin{equation*}
  \left(\frac{1}{|B|}\int_{B} |f(x)-P_{t_{B}}f(x)|^{2}dx\right)^{\frac{1}{2}}\lesssim\left(\frac{1}{|B|}\int_{B} |f(x)-f_{B}|^{2}dx+\frac{1}{|B|}\int_{B}|P_{t_{B}}(f-f_{B})(x)|^{2}dx\right)^{\frac{1}{2}}
 \lesssim_{\alpha}  t_{B}^{\alpha}\|f\|_{\mathcal{L}_{\alpha}(\mathbb{R}^{n})}.
\end{equation*}
Dividing both sides of the above inequality by $ t_{B}^{\alpha}$, and taking the supremum over $t>0$ on the left-hand side, then the desired result \eqref{dy1} follows readily.

Now we prove \eqref{dy2}. Note that for $0<\alpha<1$,
\begin{equation*}
\begin{aligned}
  t^{1-\alpha}\left|\frac{\partial P_{t}f(x)}{\partial t}\right|\leq &t^{1-\alpha}\left|\frac{\partial P_{t}}{\partial t}(f-P_{t}f)(x)\right|+t^{1-\alpha}\left|\frac{\partial P_{t}}{\partial t}P_{t}f(x)\right| \\
 =& t^{1-\alpha}\left|\int_{\mathbb{R}^{n}}\frac{\partial p_{t}}{\partial t}(x-y)(f(y)-P_{t}f(y))dy\right| +t^{1-\alpha}\left|\left.\frac{\partial P_{s}f(x)}{\partial s}\right|_{s=2t}\right|\\
 \leq &  t^{1-\alpha}\left(\int_{\mathbb{R}^{n}}\left|\frac{\partial p_{t}}{\partial t}(x-y)\right|dy \int_{\mathbb{R}^{n}}\left|\frac{\partial p_{t}}{\partial t}(x-y)\right||f(y)-P_{t}f(y)|^{2}dy\right)^{\frac{1}{2}}
 \\&+2^{\alpha-1}(2t)^{1-\alpha}\left|\left.\frac{\partial P_{s}f(x)}{\partial s}\right|_{s=2t}\right|.
 \end{aligned}
\end{equation*}
Since $$\int_{\mathbb{R}^{n}}\left|\frac{\partial p_{t}}{\partial t}(y)\right|dy\lesssim\frac{1}{t}\quad\text{and}\quad
\left|\frac{\partial p_{t}}{\partial t}(y)\right|\lesssim\frac{1}{t}p_{t}(y),$$
we then get
\begin{equation*}
\begin{aligned}
  t^{1-\alpha}\left|\frac{\partial P_{t}f(x)}{\partial t}\right|\lesssim & t^{-\alpha}\left(\int_{\mathbb{R}^{n}}\left|p_{t}(x-y)|f(y)-P_{t}f(y)\right|^{2}dy\right)^{\frac{1}{2}} +2^{\alpha-1}\sup_{t>0}t^{1-\alpha}\left\|\frac{\partial P_{t}f}{\partial t}\right\|_{\infty} \\
  \leq & \|f\|_{\mathcal{\dot{L}}_{\alpha}(\mathcal{P})} +2^{\alpha-1}\sup_{t>0}t^{1-\alpha}\left\|\frac{\partial P_{t}f}{\partial t}\right\|_{\infty}.
\end{aligned}
\end{equation*}
Taking the supremum over $x\in\mathbb{R}^{n}$ and $t>0$ on the left-hand side, and using the fact that $2^{\alpha-1}<1$ when $0<\alpha<1$, we obtain
$$\|f\|_{\dot{\Lambda}_{\alpha}(\mathcal{P})}\lesssim\|f\|_{\mathcal{\dot{L}}_{\alpha}(\mathcal{P})}.$$}
\end{remark}

\bigskip

Now we are going to show the desired equivalences between the spaces $\mathcal{L}_{\alpha}(\mathcal{P})$, $\mathcal{L}_{\alpha}(\mathcal{H})$, and their little versions (cf. \cite[Section 6]{HJ24}).   Let $0<\alpha<1$. The spaces $\mathcal{\ell}_{\alpha}(\mathcal{P})$ and $\mathcal{\ell}_{\alpha}(\mathcal{H})$ are defined respectively as the subspace of $L_\infty(\mathbb R^n)$ with finite norms
\begin{equation}\label{little1}
\|f\|_{\mathcal{\ell}_{\alpha}(\mathcal{P})}=\|f\|_{\infty}+\sup_{t>0}\frac{1}{t^{\alpha}}\left\|e^{-t\sqrt{\Delta}}|f|^{2}-|e^{-t\sqrt{\Delta}}f|^{2}\right\|^{\frac{1}{2}}_{\infty},
\end{equation}
\begin{equation}\label{little2}
\|f\|_{\mathcal{\ell}_{\alpha}(\mathcal{H})}=\|f\|_{\infty}+\sup_{t>0}\frac{1}{t^{\alpha}}\left\|e^{-t\Delta}|f|^{2}-|e^{-t\Delta}f|^{2}\right\|^{\frac{1}{2}}_{\infty}.
\end{equation}
Similarly, one can denote the homogeneous part of $\|\cdot\|_{\mathcal{\ell}_{\alpha}(\mathcal{P})}$ and $\|\cdot\|_{\mathcal{\ell}_{\alpha}(\mathcal{H})}$ by  $\|\cdot\|_{\dot{\mathcal{\ell}}_{\alpha}(\mathcal{P})}$ and $\|\cdot\|_{\dot{\mathcal{\ell}}_{\alpha}(\mathcal{H})}$, respectively.
\begin{proposition}\label{c2c}
Let  $(H_{t})_{t>0}$ and $(P_{t})_{t>0}$ be as above. For $0<\alpha<1/2$, the following equivalences hold:
\begin{itemize}
  \item [(i)] $\mathcal{\ell}_{\alpha}(\mathcal{P})\simeq\mathcal{L}_{\alpha}(\mathcal{P});$
  \item [(ii)] $\mathcal{\ell}_{\alpha}(\mathcal{H})\simeq\mathcal{L}_{\alpha}(\mathcal{H}).$
\end{itemize}
\end{proposition}

\begin{proof}
(i) Using the triangle inequality, we have
\begin{equation}\label{diyi}
\begin{aligned}
  \left|e^{-t\sqrt{\Delta}}|f-e^{-t\sqrt{\Delta}}f|^{2}(x)\right|^{\frac{1}{2}}=& \left|\int_{\mathbb{R}^{n}}p_{t}(x-y)|f(y)-e^{-t\sqrt{\Delta}}f(y)dy|^{2}\right|^{\frac{1}{2}} \\
 \leq& \left|\int_{\mathbb{R}^{n}}p_{t}(x-y)|f(y)-e^{-t\sqrt{\Delta}}f(x)|^{2}dy\right|^{\frac{1}{2}}
\\&+ \left|\int_{\mathbb{R}^{n}}p_{t}(x-y)|e^{-t\sqrt{\Delta}}f(y)-e^{-t\sqrt{\Delta}}f(x)|^{2}dy\right|^{\frac{1}{2}}\\
=:&  A+B.
\end{aligned}
\end{equation}
By the Markovian property  \eqref{lL},
\begin{equation}\label{pp1}
  A=\left|e^{-t\sqrt{\Delta}}|f|^{2}(x)-|e^{-t\sqrt{\Delta}}f(x)|^{2}\right|^{\frac{1}{2}}\leq t^{\alpha}\|f\|_{\dot{\mathcal{\ell}}_{\alpha}(\mathcal{P})}.
\end{equation}

For $B$, by the Cauchy-Schwarz inequality and \eqref{lL} again, one has
\begin{equation*}
\begin{aligned}
  |P_{t}f(y)-P_{t}f(x)|^{2}= &\left|\int_{\mathbb{R}^{n}}p_{t}(y-z)(f(z)-e^{-t\sqrt{\Delta}}f(x))dz\right|^{2} \\
  \leq & \int_{\mathbb{R}^{n}}p_{t}(y-z)dz\int_{\mathbb{R}^{n}}p_{t}(y-z)|f(z)-e^{-t\sqrt{\Delta}}f(x)|^{2}dz \\
  =&\int_{\mathbb{R}^{n}}p_{t}(y-z)|f(z)-e^{-t\sqrt{\Delta}}f(x)|^{2}dz.
 \end{aligned}
\end{equation*}
One then gets by the contraction and Markovian property as well as the kernel estimate $p_{2t}\leq 2p_t$,
\begin{equation*}
\begin{aligned}
  B\leq & \left|\int_{\mathbb{R}^{n}}p_{t}(x-y)\int_{\mathbb{R}^{n}}p_{t}(y-z)|f(z)-e^{-t\sqrt{\Delta}}f(x)|^{2}dzdy\right|^{\frac{1}{2}} \\
  =& \left|\int_{\mathbb{R}^{n}}p_{2t}(x-z)|f(z)-e^{-t\sqrt{\Delta}}f(x)|^{2}dz\right|^{\frac{1}{2}} \\
   \leq& \left|2\int_{\mathbb{R}^{n}}p_{t}(x-z)|f(z)-e^{-t\sqrt{\Delta}}f(x)|^{2}dz\right|^{\frac{1}{2}}\\
  \leq& \sqrt{2}t^{\alpha}\|f\|_{\mathcal{\ell}_{\alpha}(\mathcal{P})}.
 \end{aligned}
\end{equation*}
Combining the estimates of the terms $A$ and $B$, and then taking the supremum over $t>0$ on the left-hand side of \eqref{diyi}, one gets
\begin{equation}\label{poi1}
\|f\|_{\mathcal{L}_{\alpha}(\mathcal{P})}= \sup_{t>0}\frac{1}{t^{\alpha}}\left\|P_{t}|f-P_{t}f|^{2}\right\|^{\frac{1}{2}}_{\infty}\leq (\sqrt{2}+2)\|f\|_{\mathcal{\ell}_{\alpha}(\mathcal{P})}.
\end{equation}

For the converse direction, by \eqref{pp1}, one has
$$\|f\|_{\dot{\mathcal{\ell}}_{\alpha}(\mathcal{P})}=\sup_{t>0,x\in\mathbb{R}^{n}}\frac{1}{t^{\alpha}}\left|P_{t}|f|^{2}(x)-|P_{t}f(x)|^{2}\right|^{\frac{1}{2}}
=\sup_{t>0,x\in\mathbb{R}^{n}}\frac{1}{t^{\alpha}}\left|\int_{\mathbb{R}^{n}}p_{t}(x-y)|f(y)-P_{t}f(x)|^{2}dy\right|^{\frac{1}{2}}.$$
Let $B_{t}=B(x,t)$ and $f_{B_{t}}=\frac{1}{|B_{t}|}\int_{B_{t}}f(y)dy$. We then use the fact that the series $\sum_{k=0}^{\infty}2^{k(2\alpha-1)}$ is convergent when $0<\alpha<\frac{1}{2}$ to get
\begin{equation}\label{poi2}
\begin{aligned}
 &\int_{\mathbb{R}^{n}}p_{t}(x-y)|f(y)-f_{B_{t}}|^{2}dy & \\=&\int_{B_{t}}\frac{c_{n}t}{(t^{2}+|x-y|^{2})^{\frac{n+1}{2}}}|f(y)-f_{B_{t}}|^{2}dy \\
  +&\sum_{k=1}^{\infty}\int_{2^{k}B_{t}\setminus2^{k-1}B_{t}}\frac{c_{n}t}{(t^{2}+|x-y|^{2})^{\frac{n+1}{2}}}|f(y)-f_{B_{t}}|^{2}dy \\
\lesssim & \frac{1}{t^{n}}\int_{B_{t}}|f(y)-f_{B_{t}}|^{2}dy
+2^{n+1}\sum_{k=1}^{\infty}2^{-k}\frac{1}{(2^{k}t)^{n}}\int_{2^{k}B_{t}}|f(y)-f_{B_{t}}|^{2}dy\\
\lesssim&(1+2^{n+1}\sum_{k=1}^{\infty}2^{k(2\alpha-1)})t^{2\alpha}\|f\|_{\mathcal{\dot{L}}_{\alpha}(\mathbb{R}^{n})}\\ \lesssim& t^{2\alpha}\|f\|_{\mathcal{\dot{L}}_{\alpha}(\mathbb{R}^{n})}^{2}.
\end{aligned}
\end{equation}
By making use of Theorem \ref{c1c}, one deduces that
\begin{equation}\label{pp2}
\begin{aligned}
 &\int_{\mathbb{R}^{n}}p_{t}(x-y)|f(y)-P_{t}f(x)|^{2}dy\\ \leq&2\int_{\mathbb{R}^{n}}p_{t}(x-y)|f(y)-f_{B_{t}}|^{2}dy
  +2\int_{\mathbb{R}^{n}}p_{t}(x-y)|f_{B_{t}}-P_{t}f(x)|^{2}dy\\
  \eqref{lL}=&2\int_{\mathbb{R}^{n}}p_{t}(x-y)|f(y)-f_{B_{t}}|^{2}dy+2|f_{B_{t}}-P_{t}f(x)|^{2}\\
  \lesssim & t^{2\alpha}\|f\|_{\mathcal{\dot{L}}_{\alpha}(\mathbb{R}^{n})}^{2}
  \\ \simeq& t^{2\alpha}\|f\|_{\mathcal{\dot{L}}_{\alpha}(\mathcal P)}^{2},
\end{aligned}
\end{equation}
since
\begin{equation*}
\begin{aligned}
  |f_{B_{t}}-P_{t}f(x)|^{2}= & |f_{B_{t}}- \int_{\mathbb{R}^{n}}p_{t}(x-y)f(y)dy|^{2} \\
\eqref{lL}= &|\int_{\mathbb{R}^{n}}p_{t}(x-y)(f(y)-f_{B_{t}})dy|^{2}
  \\ \leq&\int_{\mathbb{R}^{n}}p_{t}(x-y)|f(y)-f_{B_{t}}|^{2}dy.
\end{aligned}
\end{equation*}
Dividing both sides of \eqref{pp2} by $t^{2\alpha}$, and then taking the supremum over $t>0$ on the left-hand side of \eqref{pp2}, one has
\begin{equation}\label{pp3}
  \|f\|_{\dot{\mathcal{\ell}}_{\alpha}(\mathcal{P})}\lesssim \|f\|_{\mathcal{\dot{L}}_{\alpha}(\mathcal P)}.
\end{equation}
Therefore, integrating \eqref{poi1} with \eqref{pp3} yields (i).

(ii) Repeating the argument in (i), with $B_t=B(x,t)$ replaced by $B_{\sqrt{t}}=B(x,\sqrt{t})$ and \eqref{poi2} replaced by the following inequality,
\begin{equation*}
\begin{aligned}
  \int_{\mathbb{R}^{n}}h_{t}(x-y)|f(y)-f_{B_{\sqrt{t}}}|^{2}dy \lesssim &\int_{B_{\sqrt{t}}}t^{-\frac{n}{2}}\left(1+\frac{|x-y|^{2}}{t}\right)^{-(n+2\alpha+1)}|f(y)-f_{B_{\sqrt{t}}}|^{2}dy \\
  \\&+\sum_{k=1}^{\infty}\int_{2^{k}B_{\sqrt{t}}\setminus2^{k-1}B_{\sqrt{t}}}t^{-\frac{n}{2}}\left(1+\frac{|x-y|^{2}}{t}\right)^{-(n+2\alpha+1)}|f(y)-f_{B_{\sqrt{t}}}|^{2}dy \\
\lesssim & \frac{1}{\sqrt{t}^{n}}\int_{B_{\sqrt{t}}}|f(y)-f_{B_{\sqrt{t}}}|^{2}dy
\\&+2^{2(n+2\alpha+1)}\sum_{k=1}^{\infty}2^{-k(n+4\alpha+2)}\frac{1}{(2^{k}\sqrt{t})^{n}}\int_{2^{k}B_{\sqrt{t}}}|f(y)-f_{B_{\sqrt{t}}}|^{2}dy\\
\lesssim&(1+2^{2(n+2\alpha+1)}\sum_{k=1}^{\infty}2^{-k(n+2)})t^{2\alpha}\|f\|_{\mathcal{\dot{L}}_{2\alpha}(\mathbb{R}^{n})}^{2}
\\ \lesssim&_{n,\alpha}t^{2\alpha}\|f\|_{\mathcal{\dot{L}}_{2\alpha}(\mathbb{R}^{n})}^{2}
\\ \simeq& t^{2\alpha}\|f\|_{\mathcal{\dot{L}}_{\alpha}(\mathcal{H})}^{2}.
\end{aligned}
\end{equation*}
\end{proof}

\section{The case of Ornstein-Uhlenbeck semigroups}
In this part, we discuss the Gauss Campanato spaces associated with Ornstein-Uhlenbeck semigroups. Let $-L=\frac{1}{2}\sum_{i=1}^{n}\frac{\partial^{2}}{\partial x_{i}^{2}} - \sum_{i=1}^{n}x_i\cdot\frac{\partial}{\partial x_{i}}$ on $\bigl(\mathbb R^n,d\gamma\bigr)$ with $d\gamma=\pi^{-n/2}e^{-|x|^2}dx$, which is slightly different from the Gaussian measure  $(2\pi)^{-n/2}e^{-\frac{|x|^2}{2}}dx$ commonly used in probability theory.
This operator defined above is the infinitesimal generator of the Ornstein-Uhlenbeck semigroup $O_{t}=e^{-tL}$ $(t>0)$, given by
$$
O_t f(x)
=\frac{1}{\pi^{n / 2}\left(1-e^{-2 t}\right)^{n / 2}}\int_{\mathbb R^n}e^{-\frac{\left|y-e^{-t} x\right|^2}{1-e^{-2 t}}}f(y)dy.
$$
The Ornstein-Uhlenbeck semigroup $(O_{t})_{t>0}$ is strongly continuous on $L^{1}(\mathbb{R}^{n},d\gamma)$ (cf. \cite[Theorem 2.5]{uw19}), it then follows from \cite[Page 113]{h87} that $(O_{t})_{t>0}$ is weak-$*$ continuous on $L^{\infty}(\mathbb{R}^{n},d\gamma)$. However, the Ornstein-Uhlenbeck semigroup $(O_{t})_{t>0}$ does not meet the analyticity assumption in the present paper. Indeed, \cite[Section 5]{mpp02} shows that $(O_{t})_{t>0}$ fails to be analytic on $L^1(\mathbb{R}^n,d\gamma)$;  by \cite[Lemma 1.10.1 and 1.10.2]{Pazy83}, it therefore cannot be analytic on the dual space $L^\infty(\mathbb{R}^n,d\gamma)$.

Let $\mathcal{U}=(U_{t})_{t>0}$ be the Ornstein-Uhlenbeck Poisson semigroup subordinated to $(O_{t})_{t>0}$. As noted in Remark \ref{jiajia}, the Ornstein-Uhlenbeck Poisson semigroup $(U_t)_{t>0}$ is analytic on
$L^{\infty}(\mathbb R^n, d\gamma)$. Consequently, our main result-Theorem \ref{e25}, directly applies to the Ornstein-Uhlenbeck Poisson semigroup $(U_t)_{t>0}$.
In other words, for $\alpha>0$ and $f\in L^{\infty}(\mathbb R^n, d\gamma)$, one has
$$
\|f\|_{\dot{\mathcal{L}}_{\alpha}(\mathcal{U})}:=  \sup_{t>0}\frac{1}{t^\alpha}
\bigl\|U_t\bigl|(I-U_t)^{[\alpha]+1}f\bigr|^{2}\bigr\|^{\frac{1}{2}}_\infty
\simeq_{\alpha} \sup_{t>0}\frac{1}{t^{\alpha-([\alpha]+1)}}
\Bigl\|\frac{\partial^{[\alpha]+1}U_{t}f}{\partial t^{[\alpha]+1}}\Bigr\|_{\infty}=:\|f\|_{\dot{\Lambda}_{\alpha}(\mathcal{U})}.$$
Moreover, using the fact that the family $(U_t)_{t>0}$ is a Markov semigroup on $(\mathbb{R}^n,d\gamma)$ satisfying the $\Gamma^{2}\geq 0$ criterion (see e.g. \cite{tms19}), it follows immediately from \cite[Proposition 6.7]{HJ24} that for $0<\alpha<1/2$,
$$\|f\|_{\dot{\mathcal{L}}_{\alpha}(\mathcal{U})}\simeq_{\alpha}\|f\|_{\dot{\ell}_{\alpha}(\mathcal{U})}:=\sup_{t>0}\frac{1}{t^{\alpha}}\left\|e^{-t\sqrt{L}}|f|^{2}-|e^{-t\sqrt{L}}f|^{2}\right\|^{\frac{1}{2}}_{\infty}.$$

In a manner similar to that discussed in Theorem \ref{c1c}(i), naturally one also concern the difference characterization of $\dot{\Lambda}_{\alpha}(\mathcal{U})$. Indeed, in \cite{ls16, ls17}, Liu and Sj$\ddot{\mathrm{o}}$gren proved that for $0<\alpha<1$ and $f\in L^{\infty}(\mathbb R^n,d\gamma)$,
$$\|f\|_{\dot{\Lambda}_{\alpha}(\mathcal{U})}\simeq_{\alpha}\|f\|_{\dot{\Lambda}_{\alpha}(\gamma)}:=\sup_{x,y\in \mathbb R^n, x\neq y}\frac{|f(x)-f(y)|}{\min{\left\{|x-y|^{\alpha},(\frac{|x-y|}{1+|x|+|y|})^{\frac{\alpha}{2}}+((|x|+|y|)sin\theta)^{\alpha}\right\}}},$$
where $\theta$ is the angle between the vectors $x$ and $y$, with the convention
$\theta = 0$ whenever $x = 0$ or $y = 0.$ A difference characterization for $\alpha\geq1$ seems unavailable in the literature.

Motivated by \cite{ls16, ls17}, Liu and Yang \cite[Proposition 2.3]{yl14} established a difference characterization of the Gauss Campanato space defined geometrically. Given $0<\alpha<1$ and $f\in L^{\infty}(\mathbb R^n,d\gamma)$, the semi-norm of the homogeneous Gauss Campanato space is defined by
\begin{equation}\label{qqqqqq}
\|f\|_{\dot{\mathcal{L}}_{\alpha}(\gamma)}=\sup_{r_B\leq\min\{1,\frac{1}{|c_B|}\}}\frac{1}{(\gamma(B))^\frac{\alpha}{n}}\left\{\frac{1}{\gamma(B)}\int_{B}|f(x)- f_B^\gamma(x)|^2d\gamma\right\}^{\frac{1}{2}},
\end{equation}
where $c_B$ and $r_B$ are the center and radius of $B$ and $f_B^\gamma=\frac{1}{\gamma(B)}\int_{B}f(x)d\gamma$; and they showed that
$$\|f\|_{\dot{\mathcal{L}}_{\alpha}(\gamma)}\simeq_{\alpha}\|f\|_{\dot{\tilde{\Lambda}}_{\alpha}(\gamma)}=\sup_{r_B\leq\min\{1,\frac{1}{|c_B|}\},x,y\in B}\frac{|f(x)-f(y)|}{(\gamma(B))^{\frac{\alpha}{n}}}.$$
Again, an equivalence for $\alpha\geq1$ seems unavailable until now. On the other hand, because of the restriction $r_B\leq\min\{1,\frac{1}{|c_B|}\}$ in \eqref{qqqqqq}, it is easy to note that $\|f\|_{\dot{\Lambda}_{\alpha}(\gamma)}$ is not equivalent to $\|f\|_{\dot{\tilde{\Lambda}}_{\alpha}(\gamma)}$; and a Gauss Campanato space without the restriction  seems to have not been investigated in the literature.

Finally, inspired by Theorem \ref{c1c}(ii), one is tempted to ask for the following equivalence for $0<\alpha<1/2$ and $f\in L^{\infty}(\mathbb R^n,d\gamma)$,
$$\|f\|_{\dot{\mathcal{L}}_{\alpha}(\gamma)}\simeq_\alpha\|f\|_{\dot{\mathcal{L}}_{\alpha,\sqrt{L}}(\gamma)}:=\sup_{r_B\leq\min\{1,\frac{1}{|c_B|}\}}\frac{1}{(\gamma(B))^\frac{\alpha}{n}}\left\{\frac{1}{\gamma(B)}\int_{B}|f(x)- e^{-r_{B}\sqrt{L}}f(x)|^2d\gamma\right\}^{\frac{1}{2}}?$$

We will try to understand all this elsewhere.

\subsection*{Acknowledgments}
The authors are grateful to the referee for valuable comments and suggestions that helped improve the presentation of this paper. This work is partially supported by the National
Natural Science Foundation of China (No. 12071355,
No. 12325105, No. 12031004, No. W2441002, NO. W2611005). Ping Li is supported by NSFC (no. 12371136) and the open Research Fund of Key Laboratory of Nonlinear Analysis \& Applications (Central China Normal University), Ministry of Education, P.R. China. Yuanyuan Jing is supported by the Postdoctoral Fellowship Program of CPSF under Grant Number GZC20252763 and the China Postdoctoral Science Foundation under Grant Number 2025M784342.



\end{document}